\documentclass[reqno,10pt]{amsart}

\usepackage{amsmath, amsthm, amsfonts}
\usepackage{pdfsync}
\usepackage[pagewise]{lineno}%\linenumbers
\usepackage{parskip}
\usepackage{hyperref}
\usepackage{mathrsfs}  
\usepackage{mathtools}
\usepackage{mdframed}
\usepackage{enumerate}
\usepackage{enumitem}
\usepackage{tikz}
\numberwithin{equation}{section}
\newtheorem{theorem}{Theorem}[section]
\newtheorem{prop}[theorem]{Proposition}
\newtheorem{definition}[theorem]{Definition}
\newtheorem{lem}[theorem]{Lemma}

\theoremstyle{remark}
\newtheorem{remark}[theorem]{Remark}

\def\M{\mathsf{M}}

\def\R{\mathbb{R}}

\def\N{\mathbb{N}}

\def\cA{\mathcal{A}}

\def\id{{\rm{id}}}

\def\div{{\rm div}}

\def\max{{\rm max}}

\begin{document}

\title[Functional Inequalities and Fractional Porous Medium Equation]{Functional Inequalities involving Nonlocal Operators on Complete Riemannian Manifolds and Their Applications to The Fractional Porous Medium Equation}

\author[N. Roidos]{Nikolaos Roidos}
\address{Department of Mathematics, University of Patras, 26504 Rio Patras, Greece}
\email{roidos@math.upatras.gr}

\author[Y. Shao]{Yuanzhen Shao}
\address{Department of Mathematics,
         The University of Alabama,
         Box 870350,
         Tuscaloosa, AL 35487-0350,
         USA}
\email{yshao8@ua.edu}


\begin{abstract}
The objective of this paper is twofold. First,  we conduct a careful study of various functional inequalities   involving the fractional Laplacian operators, including nonlocal Sobolev-Poincar\'e, Nash, Super Poincar\'e and logarithmic Sobolev type inequalities, on complete Riemannian manifolds satisfying some mild geometric assumptions.
Second, based on the derived nonlocal functional inequalities, we analyze the asymptotic behavior  of the solution  to the fractional porous medium equation, $\partial_t u +(-\Delta)^\sigma (|u|^{m-1}u  )=0 $ with $m>0$ and $\sigma\in (0,1)$.
In addition, we establish the global well-posedness of the equation on an arbitrary complete Riemannian manifold.
\end{abstract}

\maketitle

\section{\bf Introduction}\label{Section 1}

This manuscript is mainly motivated by the asymptotic behavior of the solution to  the following fractional porous medium equation:
\begin{equation}
\label{S1: FPME}
\left\{\begin{aligned}
\partial_t u +(-\Delta )^\sigma (|u|^{m-1}u  )&=0   &&\text{on}&&\M\times (0,\infty);\\
u(0)&=u_0   &&\text{on}&&\M
\end{aligned}\right.
\end{equation}
for $m \in (0, \infty)$ and $\sigma\in (0,1)$  
on a  smooth complete Riemannian manifold $(\M,g)$ without boundary.
Here $\Delta $ is the Laplace-Beltrami operator associated with the Riemannian metric $g$.

Global well-posedness of the Cauchy problem of \eqref{S1: FPME} was first studied by A. Pablo, F. Quir\'os, A. Rodr\'iguez and J.L. V\'azquez in \cite{PalRodVaz11, PalRodVaz12} in Euclidean spaces, and later by M. Bonforte, A. Pablo, F. Quir\'os,  A. Rodr\'iguez,  Y. Sire  and J.L. V\'azquez in \cite{BonSireVaz15, BonVaz15, BonVaz16, PalRodVaz12} for the Dirichlet problem on bounded domains. 
Since then, there has been a vast amount of work \cite{BonFigOto17, BonFigVaz18, Vaz14, Vaz15}, just to name a few, investigating various properties, e.g. regularity and Barenblatt solutions, of the solutions to \eqref{S1: FPME}.
%Interested readers may refer to \cite{BonFigOto17, BonFigVaz18, Vaz14, Vaz15}, just to name a few.
See also \cite{AlpEll16, Grillo15, PunTer14} for some related work. 
Most of the work on \eqref{S1: FPME} considered flat spaces. As far as we know, the only exception is the work \cite{AlpEll16} by A. Alphonse and C.M. Elliott, which studied a variation of \eqref{S1: FPME} with fractional power $\sigma=1/2$ on a closed manifold. 
Very recently, we studied \eqref{S1: FPME}  on an incomplete manifold with isolated conical singularities and finite volume in  \cite{RoidosShao1802, RoidosShao18}.

If we come to the problem of asymptotic behavior of the solution to \eqref{S1: FPME} on Riemannian manifolds, it becomes clear  that an essential component is still missing. 
Indeed, as shown in \cite[Section~5]{PalRodVaz12},  certain nonlocal functional inequalities  play a crucial role in  the asymptotic analysis of \eqref{S1: FPME} in Euclidean spaces.
It is expected as the role of the local counterparts of these inequalities  have already been recognized in the asymptotic analysis of the porous medium equation in Euclidean spaces.  

In manifolds setting, most of the research \cite{BonGriVaz08, Grillo20, GrilloMur14, GrilloMur16, GrilloMurPun18, GrilloMurPun1802, GrilloMurVaz17, GrilloMurVaz19, Vaz1502} on large time behavior of the porous medium equation  focuses on  manifolds of nonpositive curvature, e.g. the hyperbolic space. This is mainly due to the validity of certain functional inequalities in such spaces. 
In order to study \eqref{S1: FPME} for more general manifolds setting,  we will use the idea by M. Bonforte and G. Grillo in \cite{BonGri05}, which studied the asymptotic behavior of solution to the porous medium equation  based on some logarithmic Sobolev inequalities. Such inequalities are known to hold on closed manifolds or manifolds satisfying the Faber-Krahn inequality.	
%On Riemannian manifolds, arguably, the most successful approach to the asymptotic behavior of solution to the porous medium equation is based on certain logarithmic Sobolev inequalities, c.f. \cite{BonGri05}. 
%The success of logarithmic Sobolev inequalities in the study of asymptotics of the porous medium equation is somehow expected, as 
The asymptotic behavior or the smoothing effect of the porous medium equation  can be viewed as a special case of the ultracontractivity property.
We recall that the ultracontractivity property of an equation allows the $L_\infty-$norm of the solution to be bounded by the $L_p-$norm of the initial datum.
As was first discovered in the pioneering work \cite{Gross} by L. Gross, the ultracontractivity of the heat semigroup is closely related to logarithmic Sobolev inequalities.
This relationship was later explored in more depth in the monograph \cite{Dav89} by E.~Davies. 
The proofs heavily rely on the theory of  symmetric Markov semigroups, which, briefly speaking,  are strongly continuous symmetric semigroups that are order-preserving and contractive on $L_\infty(\M)$. 
The idea of using logarithmic Sobolev inequalities to study the ultracontractive bounds and asymptotics of nonlinear evolutions traces back to the work \cite{CarLoss95} by E.~Carlen and M.~Loss.  
This method was later applied to the   $p$-Laplacian equation in \cite{CipGri01}.
Those observations suggest that a similar approach should be applicable to \eqref{S1: FPME} as well, because the associated nonlinear semigroup to \eqref{S1: FPME} satisfies some order-preserving and contractive properties. 
%Based on the nonlocal logarithmic Sobolev inequalities developed in Section~\ref{Section 4}, we show the ultracontractive property and asmyptotics of solutions to \eqref{S1: FPME} when $m>1$ in Sections~\ref{Section 6} and \ref{Section 7}.

However, the validity of  nonlocal versions of logarithmic Sobolev inequalities  remains an open question on general Riemannian manifolds.
A widely used approach to functional inequalities  on manifolds is  via a local to global argument.
After a moment of reflection, it is not hard to convince ourselves that the method does not work for our purpose as the fractional Laplacian operator is nonlocal.
To overcome the difficulty in establishing functional inequalities involving nonlocal operators like the fractional Laplacian, we take use of the subordination theory by Bochner. 
Briefly speaking, subordination is a method to construct a new semigroup from a given one. 
Particularly, we will show in Section~\ref{Section 3} that the fractional Laplacian can be constructed in terms of $\Delta$ via subordination and the associated semigroup is again Markovian; and then based on the  Markovian property, in Section~\ref{Section 4}, we establish two versions of nonlocal logarithmic Sobolev inequality as well as many other functional inequalities involving the fractional Laplacian under some mild geometric conditions, cf. Section~\ref{Section 2}.
%including but not limited to nonlocal versions of Sobolev-Poincar\'e and Nash type inequalities.
Based on these inequalities, we further derive some useful heat kernel and semigroup estimates for the fractional Laplacian.

In the second part of the paper, we conduct a careful  study of \eqref{S1: FPME}. The emphasis is put on the asymptotic analysis of the solution.
Based on the work in Sections~\ref{Section 3} and \ref{Section 4}, we reveal a connection between nonlocal  logarithmic Sobolev inequalities and the ultracontractivity of the fractional porous medium equation~\eqref{S1: FPME}.
$L_p-L_\infty$ regularizing effects of the form $\|u(t)\|_\infty\leq C(u_0)t^{-\alpha}$ has been shown, where $C(u_0)$ depends on $\|u_0\|_p$. 
On Riemannian manifolds satisfying a Faber-Krahn type condition, we show that the solution converges to zero uniformly.
On a compact manifold  without boundary, the $L_p-L_\infty$ regularizing effects give  an $L_\infty$-bound of the solution for small $t>0$. 
Then in conjunction with an analysis of the structure of the $\omega-$limit set of the trajectory, the derived  bound implies the convergence of the solution to the mean of the initial datum. 
The asymptotic analysis is discussed in details in Sections~\ref{Section 6} and \ref{Section 7}.

Last but not least, we would like to highlight    the global well-posedness result of \eqref{S1: FPME}  we obtained in Section~\ref{Section 5} a little. 
In an Euclidean space $\R^N$,  \eqref{S1: FPME} was studied via one of the following constructions of  the fractional Laplacian or their analogues.

(1) The authors of \cite{PalRodVaz11, PalRodVaz12} constructed the fractional Laplacian via the Caffarelli-Silvestre extension \cite{CafSil07}, i.e. consider the solution of the problem:
$$
\nabla \cdot (y^{1-2\sigma} \nabla w)=0 \quad  (x,y)\in \R^N\times \R_+; \quad w(x,0)=u(x),\quad x\in \R^N.
$$
Then  for some constant $C_\sigma$
$$
(-\Delta)^\sigma u(x)= -C_\sigma \lim\limits_{y\to 0} y^{1-2\sigma}\frac{\partial w}{\partial y}(x,y).
$$

(2) In  \cite{BonSireVaz15, BonVaz15, BonVaz16}, the authors used the Spectral Fractional Laplacian (SFL) for the Dirichlet Laplacian $\Delta$ in a bounded domain $\Omega$ defined by:
$$
(-\Delta)^\sigma u (x):= \frac{1}{\Gamma(-\sigma)}\int_0^\infty (e^{t\Delta}u(x)-u(x))\,\frac{dt}{t^{1+\sigma}}= \sum\limits_{j=1}^\infty \lambda_j^\sigma \hat{u}_j \phi_j(x),
$$
where $(\phi_k, \lambda_k)_{k=1}^\infty$ is an orthonormal basis of
$L_2(\Omega)$ consisting of eigenfunctions of  $-\Delta$  and their corresponding eigenvalues.
$\hat{u}_k$ are the Fourier coefficients of $u$.

(3) In the Restricted Fractional Laplacian (RFL) approach, the fractional Laplacian is constructed by using the integral representation in terms of hypersingular kernels:
$$
(-\Delta)^\sigma u (x)= C_{N,\sigma} P.V.\int_{\R^N} \frac{u(x) - u(y)}{|x+y|^{N+2\sigma}}\, dy \quad \text{with } u(x)=0 \text { for } x\notin \Omega.
$$

(RFL) is obviously not applicable to general manifolds. 
In \cite{StinTorr10}, the authors proved that the Caffarelli-Silvestre extension holds when $\Delta$ has discrete spectrum. 
Similar results were established in \cite{BanGonSae15} under certain geometric assumptions, c.f. \cite[Proposition~3.3]{BanGonSae15}. 
However,  when such conditions are absent,  the applicability of the Caffarelli-Silvestre extension to \eqref{S1: FPME}   remains unknown. For a similar reason,   (SFL)   also seems to be restrictive.
% when the underlying space is an arbitrary Riemannian manifold.

In this paper, we will generalize the method in our earlier work \cite{RoidosShao18} to manifolds with infinite volumes.
The approach   relies only on the existence of a Markovian extension of the fractional Laplacian.
The only geometric assumption we   impose for the  global well-posedness    of \eqref{S1: FPME} is $(\M,g)$ being a complete Riemannian manifold. Particularly, no compactness, curvature or volume condition is needed.
This seems to be the most general one so far.

Before finishing the introduction, we would like to mention a very recent work \cite{BerBonGanGri} recommended to us by  G. Grillo. It studies the smoothing effect of solutions to a larger class of data  on the hyperbolic space. Their argument relies on a different method and a nonlocal  Poincar\'e inequality. 
This seems to be a very interesting direction to explore in the future.

%%%%%%%%%%%%%%%%%%%%%%%%%%%%%%%%%%%%%
\textbf{Notations:} 
%Given any interval $I$ containing $0$, $\dot{I}:= I \setminus\{0\}$. 
%Given any topological set $U$, $\mathring{U}$ denotes the interior of $U$. 
%If $U$ consists of only one point, we set $\mathring{U}:=U$. 

%%%%%%%%%%%%%%%%%%%%%%%%%%%%%%%%%%%%%%%%%
For any two Banach spaces $X,Y$, $X\doteq Y$ means that they are equal in the sense of equivalent norms. The notations
$$
X\hookrightarrow Y, \qquad X\xhookrightarrow{d} Y
$$
mean that $X$ is continuously embedded and  densely embedded, respectively.
%$\L(X,Y)$ denotes the set of all bounded linear maps from $X$ to $Y$, and $\L(X):=\L(X,X)$. Moreover, $\Lis(X,Y)$ stands for the subset of $\L(X,Y)$ consisting of all bounded linear isomorphisms from $X$ to $Y$. 
Given a sequence $(u_k)_k:=(u_1,u_2,\cdots)$ in $X$, $u_k \rightharpoonup u$ in $X$ means that $u_k$ converge weakly to some $u\in X$.
Given a densely-defined operator $\cA$ in $X$, $D(\cA)$ and $Rng(A)$ stand  for the domain and range of $\cA$, respectively.
%For any Banach space $E$, we abbreviate $\F(\R^m,E)$ to $\F(E)$ with $\F$ stands for any function space defined in this article. 
%The precise definitions for these function spaces will be presented in Section 2. 
%\smallskip\\
%%%%%%%%%%%%%%%%%%%%%%%%%%%%%%%%%%%%%%%%%

%$u\sim v$ means two functions $u,v$ are Lipschitz equivalent.

%In addition, $\R_+:=[0,\infty)$ and $\Nz:=\N\cup \{0\}$.

\section{\bf Main Results and Geometric Assumptions}\label{Section 2}

In this section, we will collect and state the main results of the article and the geometric assumptions needed for the proofs of the functional inequalities and the asymptotic behaviors of solutions to \eqref{S1: FPME}.

To prove the functional inequalities mentioned in the introduction,
we assume that $(\M,g)$ satisfies either of the following conditions.
\begin{itemize}
\item[(A1)] $(\M,g)$ is a complete and non-compact Riemannian manifolds with infinite volume and without boundary that satisfies a Faber-Krahn type condition. More precisely, there exist constant $M>0$ and  $n>2$ such that for each $\Omega\subset\subset \M$, i.e. $\overline{\Omega}$ is a compact subset of $\M$,
$$
\lambda_1(\Omega) \geq M |\Omega|^{-2/n}.
$$
Here $\lambda_1(\Omega)$ is the first Dirichlet eigenvalue of the Laplace-Beltrami operator on $\Omega$ with vanishing Dirichlet boundary condition on $\partial\Omega$.
\item[(A2)] $(\M,g)$ is an $n$-dimensional compact Riemannian manifold  without boundary for some integer $n>2$.
\end{itemize}
In this article, a closed  manifold always means one satisfying (A2).

\begin{theorem}\label{Thm: functional ineq A1}
Suppose that $(\M,g)$ is a Riemannian manifold satisfying {\em (A1)} with $n>2$. Then the followings hold  true.
\begin{itemize} 
\item[{\em (i)}] {\em (Heat kernel Gaussian upper bound):}  For some $C,A>0$, the heat kernel 
satisfies
$$
p_\sigma(x,y,t) \leq C t^{-n/2\sigma} e^{-d^2(x,y)/At},\quad x,y\in \M,\, t>0,
$$
that is,
$$
u(t,x)=\int_M p_\sigma(x,y,t)f(y)\, d\mu_g(y)
$$
solves
\begin{equation*}
\left\{\begin{aligned}
\partial_t u +(-\Delta )^\sigma u&=0   &&\text{on}&&\M\times (0,\infty);\\
u(0,x)&=f(x)   &&\text{on}&&\M.
\end{aligned}\right.
\end{equation*}
Here $d(x,y)$ is the geodesic distance between $x$ and $y$.
\item[{\em (ii)}] {\em (Ultracontractivity):} For some $C,A >0$,
$$
\|e^{-t (-\Delta)^\sigma} u \|_\infty \leq C t^{-\frac{n}{2 p\sigma }} e^{\frac{-d^2(x,y)}{Apt}}\|u\|_p,\quad 1\leq p<\infty.
$$
\item[{\em (iii)}] {\em (Nash inequality): } For some $C>0$,
$$
\|u\|^{1+(2\sigma/n)}_2 \leq C \|u\|_1^{2\sigma/n} \|( -\Delta)^{\sigma/2} u \|_2 , \quad u \in D((-\Delta)^{\sigma/2}) \cap L_1(\M).
$$
\item[{\em (iv)}] {\em (Sobolev-Poincar\'e  inequality):}  For some $\widehat{C}>0$,
$$
\|u\|_{2n/(n-2\sigma)} \leq \widehat{C} \|( -\Delta)^{\sigma/2} u\|_2 , \quad u \in D((-\Delta)^{\sigma/2}).
$$
\item[{\em (v)}] {\em (Super Poincar\'e inequality):} For any $r>0$,
$$
\|u\|_2^2 \leq r \|u \|_1^2 + \beta(r) \| (-\Delta)^{\sigma/2} u \|_2^2, \quad u \in D((-\Delta)^{\sigma/2}) \cap L_1(\M),
$$
where $\beta: (0,\infty)\to (0,\infty)$ is a decreasing function.
\item[{\em (vi)}] {\em (Logarithmic Sobolev inequality):} For all   $u\in D((-\Delta)^{\sigma/2})$ and $\varepsilon>0$,
$$
\int_\M |u|^2 \ln(\frac{|u|}{\|u\|_2 })^2\, d\mu_g \leq  \frac{n}{2\sigma}( \|u\|_2^2 \ln (\frac{1}{\varepsilon}) + \widehat{C}\varepsilon \|( -\Delta)^{\sigma/2} u\|_2^2 ),  
$$
where $\widehat{C}$ is the constant in {\em (iv)} and  $d\mu_g$ is the volume element induced by $g$.
\end{itemize}
\end{theorem}

\begin{theorem}\label{Thm: functional ineq A2}
Suppose that $(\M,g)$ is an $n$-dimensional Riemannian manifold satisfying {\em (A2)} with $n\geq 1$. Then the followings hold true.
\begin{itemize} 
\item[{\em (i)}] {\em (Heat kernel Gaussian upper bound):} For some $C,A>0$,  the heat kernel satisfies
$$
p_\sigma(x,y,t) \leq C \max\{1,t^{-n/2\sigma}\}  e^{-d^2(x,y)/At},\quad x,y\in \M,\, t>0.
$$
\item[{\em (ii)}]{\em  (Ultracontractivity):} For some $C,A >0$,
$$
\|e^{-t (-\Delta)^\sigma} u \|_\infty \leq C \max\{1, t^{-\frac{n}{2 p\sigma }} \} e^{\frac{-d^2(x,y)}{Apt}}\|u\|_p, \quad 1\leq p<\infty.
$$
\item[{\em (iii)}] {\em (Nash inequality):} When $n\geq 3$, for some $C >0$,
$$
\|u-\overline{u}\|_2^{1+2\sigma/n} \leq C 2^{1+2\sigma/n} \| (-\Delta)^{\sigma/2} u \|_2 \| u-\overline{u}\|_1^{2\sigma/n}, \quad u \in D((-\Delta)^{\sigma/2}),
$$
where $\displaystyle \overline{u}:= \frac{1}{{\rm vol}(\M) } \int_M u \, d\mu_g$   with   ${\rm vol}(\M)$   being the total volume of $(\M,g)$.
\item[{\em (iv)}] {\em (Sobolev-Poincar\'e  inequality):} When $n\geq 3$, for some $\widetilde{C} >0$,
$$
\| u -\overline{u}\|_{2n/(n-2\sigma)}  \leq \widetilde{C}\| (-\Delta)^{\sigma/2} u \|_2  , \quad u \in D((-\Delta)^{\sigma/2}).
$$
\item[{\em (v)}] {\em (Super Poincar\'e inequality ):} When $n\geq 3$, for any $r>0$
$$
\|u-\overline{u}\|_2^2 \leq r \|u-\overline{u}\|_1^2 + \beta(r) \| (-\Delta)^{\sigma/2} u \|_2^2, \quad u \in D((-\Delta)^{\sigma/2}),
$$
where $\beta: (0,\infty)\to (0,\infty)$ is a decreasing function.
\item[{\em (vi)}] {\em (Logarithmic Sobolev inequality):} There exist constants $M_0=M_0(\widetilde{C})>0$ and $M_1=M_1(\widetilde{C})>0$   such that for all $u\in D((-\Delta)^{\sigma/2})$ and $\varepsilon>0$
$$
\int_\M |u|^2 \ln(\frac{|u|}{\|u\|_2 })^2\, d\mu_g \leq  \frac{n}{2\sigma}( \|u\|_2^2 \ln (\frac{1}{\varepsilon}) + M_0\varepsilon \|( -\Delta)^{\sigma/2} u\|_2^2 + M_1 \varepsilon |\overline{u}|^2 ),
$$
where  $\widetilde{C}$ is the constant in {\em (iv)}.
\end{itemize}
\end{theorem}

Concerning the global well-poesdness of \eqref{S1: FPME}, the followings hold.
 
\begin{theorem}
\label{S4: global weak sol}
Suppose that $(\M,g)$ is a   complete Riemannian manifold, $\sigma\in (0,1)$ and $m>0$.
For every $u_0\in  L_1(\M) \cap L_{m+1}(\M)$, \eqref{S1: FPME} has a unique  strong solution in the sense of Definition~\ref{Def: solution}.
%If, in addition $m=1$, for any $u_0\in  L_1(\M) \cap L_\infty(\M)$
Additionally, the solution satisfies
\begin{itemize}
\item[{\em (I)} ] {\em Comparison principle: }If $u, \hat{u}$ are the unique strong solutions to \eqref{S1: FPME} with initial data $u_0, \hat{u}_0$, respectively, then $u_0\leq \hat{u}_0$ a.e. implies $u(t) \leq \hat{u}(t)$ a.e. for all $t\geq 0$.
\item[{\em (II)} ] {\em $L_p$-contraction: } If $u_0\in L_1(\M) \cap L_q(\M)$ with $q\in [m+1,\infty]$, then for all $0 \leq t_1 \leq t_2$ and $1\leq p\leq q$
$$
\|u(t_2)\|_p \leq \|u(t_1)\|_p.
$$
\item[{\em (III)} ] {\em Conservation of mass: } When ${\rm vol}(\M)<\infty$, for  all $t\geq 0$, it holds  that
$$
\int_\M u(t) \, d\mu_g = \int_\M u_0 \, d\mu_g.
$$
\end{itemize}
\end{theorem}

The precise definition of strong solutions can be found in Definition~\ref{Def: solution}.

When $m>1$, we can prove  the asymptotic behavior of the solution and push the initial data to $L_1(\M) \cap L_2(\M)$ under  Assumption (A1) or (A2).

\begin{theorem}
\label{S5: Large time behavior}
Suppose that $(\M,g)$ is a  Riemannian manifold satisfying {\em (A1)}   and $m>1$. 
Then for every $  u_0\in  L_1(\M) \cap L_2(\M)$, \eqref{S1: FPME} has a unique strong solution $u$.
Furthermore, if $u_0\in L_1(\M) \cap L_p(\M)$ with  $p\in [2,\infty)$, then $u$ satisfies  
$$
\| u(t)\|_\infty \leq \frac{e^R}{t^\alpha}\|u_0\|_p^\gamma
$$
for some $R=R(p,\sigma,m,n,\widehat{C})>0$, $\displaystyle \alpha=\frac{n}{ 2\sigma p  + n(m-1)}$ and $\displaystyle \gamma=\frac{2\sigma p }{2\sigma p  + n(m-1)}$, where $\widehat{C}$ is the constant in Theorem~\ref{Thm: functional ineq A1}(iv). %\eqref{S2: Sob ineq frac}.
\end{theorem}

\begin{theorem}
\label{S6: Strong solution L 2 initial data}
Suppose that $(\M,g)$ is a  Riemannian manifold satisfying {\em (A2)}  and $m>1$. Let $m_0=\max\{m-1,1\}$.
Then for every $ u_0\in     L_2(\M)$, \eqref{S1: FPME} has a unique strong solution $u$.
Furthermore, if $u_0\in L_p(\M)$ with  $p\in [2,\infty)$, then $u$ satisfies  
$$
\| u(t)\|_\infty \leq C \frac{e^{E \|u_0\|_{m_0}^{m-1}  t }}{t^\alpha}\|u_0\|_p^\gamma
$$
for some $\alpha=\alpha(p,\sigma,m,n)>0$, $C=C(p,\sigma,m,n,M_0)>0$  and 
$\gamma=(\frac{p}{p+m-1})^{n/2\sigma}$, $E=\frac{4m M_1}{M_0} $, where $M_0,M_1$ are the constants in Theorem~\ref{Thm: functional ineq A2}(vi).
Moreover,
$$
\lim\limits_{t\to \infty} \| u(t)-  \frac{1}{{\rm vol}(\M)} \int_\M u_0 \, d\mu_g \|_q =0 ,\quad 1\leq q<\infty.
$$
In particular, when  $\int_\M u_0\, d\mu_g=0$, for any  $\varepsilon\in (0,1)$ and $t>2$, it holds
$$
\|u(t)\|_\infty \leq \frac{C \|u_0\|_p^{\varepsilon \gamma} }{[B (t-1)]^{\gamma(1-\varepsilon)/(m-1)}} 
$$
for some $B=(m,p)>0$ and $C=C(p,\sigma,m,n,M_0,M_1) >0$.
\end{theorem}

This manuscript is organized as follows.

In Section~\ref{Section 3}, we study the Markovian property of the fractional Laplacian. 
In Sections~\ref{Section 4.1} and \ref{Section 4.2}, we establish  some important inequalities for the fractional Laplacian including two nonlocal logarithmic Sobolev type inequalities. 
They are an essential ingredient of the proofs of Theorems~\ref{S5: Large time behavior} and \ref{S6: Strong solution L 2 initial data} as mentioned in the introduction.
Then in Section~\ref{Section 4.3}, we derive the remaining functional inequalities, various heat kernel and semigroup estimates for the fractional Laplacian in Theorems~\ref{Thm: functional ineq A1} and \ref{Thm: functional ineq A2}.
Section~\ref{Section 5} is devoted to the proof of Theorem~\ref{S4: global weak sol}.
Theorems~\ref{S5: Large time behavior} and \ref{S6: Strong solution L 2 initial data} are proved in Sections~\ref{Section 6} and \ref{Section 7}, respectively.
To avoid possible distractions, we collect some basic facts from the Markov  semigroup and the non-linear semigroup theories in Appendix~\ref{Section A} and prove the $m$-accretivity of a perturbed fractional Laplacian.

%%%%%%%%%%%%%%%%%%%%%%%%%%%%%%%%%%%%%%%%%%%%%%%%%%%%%%%%

\section{\bf Subordinated Semigroups and the Fractional Laplacian}\label{Section 3}

Suppose that  $(\M,g)$ is a complete Riemannian manifold. Then the Laplace-Beltrami operator
$$
\Delta u =\div \nabla u
$$
is essentially self-adjoint on $C_c^\infty(\M)$, cf. \cite[Theorem~5.2.3]{Dav89}. The unique self-adjoint extension, i.e. the Friedrichs extension, will still be denoted by $\Delta$. 
It is a well-known result that $\Delta$ generates a symmetric Markov semigroup $\{e^{t\Delta}\}_{t\geq 0}$ in $L_2(\M)$, c.f. Definition~\ref{Def: Mark semigroup} and \cite[Theorem~5.11]{Grig09}.

To introduce the semigroup generated by the fractional powers of $ -\Delta$, we will need some concepts from subordinated semigroup theory.
\begin{definition}
A Bernstein function $g$ is a smooth function $:(0,\infty)\to [0,\infty)$ such that
$$
(-1)^{n-1} g^{(n)}(x) \geq 0
$$
for all $n\in\N $ and $x>0$.
\end{definition}
Standard examples of Bernstein functions include
$$
1-e^{-x}, \quad \ln(1+x),\quad x^\sigma \text{ with }\sigma\in (0,1).
$$
Following \cite[Theorem~3.2]{SchiSongVon12}, a smooth function $g:(0,\infty)\to [0,\infty)$ is a Bernstein function iff it admits a representation:
\begin{equation}
\label{S3: Bernstein}
g(x)=a+b x + \int_0^\infty (1-e^{-tx}) \nu(dt),
\end{equation}
where $a,b \geq 0$ are constants and  $\nu$ is a measure on $(0,\infty)$ satisfying 
$$
\int_0^\infty (1 \wedge t) \nu(dt)<\infty.
$$
The  triplet $(a,b,\nu)$ defines $g$ uniquely and vice versa.
The measure $\nu$ and the triplet $(a,b,\nu)$ in \eqref{S3: Bernstein}
are called the {\em L\'evy measure} and the {\em L\'evy triplet} of the Bernstein
function $g$.  

There is another way to characterize a Bernstein function $g$. Given a  convolution semigroup of sub-probability measure $\{\mu_t\}_{t>0}$ on $[0,\infty)$. Then there exists a unique Bernstein function $g$ such that the Laplace transform of $\{\mu_t\}_{t>0}$ satisfies
\begin{equation}
\label{S3: Laplace transform}
\int_0^\infty e^{-sx} \, d\mu_t(s)= e^{-t g(x)} .
\end{equation}
Conversely, given a Bernstein function $g$, there exists a unique convolution semigroup of sub-probability measures $\{\mu_t\}_{t>0}$ on $[0,\infty)$ such that \eqref{S3: Laplace transform} holds, cf. \cite[Theorem~5.2]{SchiSongVon12}.

Suppose that  $ A: D(A)\subset L_2(\M) \to L_2(\M)$ is a non-negative self-adjoint operator, which generates a Markov  semigroup  $\{e^{-t A} \}_{t\geq 0}$. Given a Bernstein function $g$ and its corresponding convolution semigroup of sub-probability measures $\{\mu_t\}_{t>0}$, then the Bochner integral
\begin{equation}
\label{S2: def subordinate semigroup}
e^{-t g(A)} u =  \int_0^\infty e^{-sA} u\, d\mu_t(s),\quad u\in L_2(\M) 
\end{equation}
defines again a symmetric Markov semigroup, cf. \cite[Proposition~12.1]{SchiSongVon12}.
This semigroup is called {\em subordinate} (in the sense of Bochner) to the semigroup $\{e^{-t A} \}_{t\geq 0}$ with respect to the Bernstein function $g$.
Its infinitesimal generator $g(A) $ is given by the Phillips formula
$$
-g(A)= -a u +bA u + \int_0^\infty (e^{-s A}u-u)\, \nu(ds),
$$
where $(a,b,\nu)$ is the {\em L\'evy triplet} of  $g$, cf \cite[Theorem~12.6]{SchiSongVon12}.

In the sequel, we will focus on the case $g(x)=x^\sigma$  for $x>0$ and $\sigma\in (0,1)$.
This is a Bernstein function with {\em L\'evy triplet}
$$
a=b=0,\quad  \nu(dt)= \frac{\sigma}{  \Gamma(1-\sigma)} t^{-\sigma-1} dt,
$$
which gives
\begin{equation}
\label{Phillips formula}
g(-\Delta)u= ( -\Delta)^\sigma u= \frac{ \sigma}{\Gamma(1-\sigma)} \int_0^\infty (u-e^{t \Delta} u) t^{-\sigma-1} dt ,\quad u\in D(\Delta).
\end{equation}
This definition    is equivalent to Balakrishnan's  formula for the fractional power of a dissipative operator:
\begin{align}
\label{S3: Balakrishnan}
(  -\Delta)^\sigma u: = -\frac{\sin(\pi\sigma)}{\pi} \int_0^\infty t^{\sigma-1}   \Delta    (t  -\Delta)^{-1}u \, dt ,\quad u\in D(\Delta).
\end{align}
See \cite[Section 12.2]{SchiSongVon12}.

By \eqref{S2: def subordinate semigroup}, we immediately obtain the following proposition concerning $( -\Delta)^\sigma$.
\begin{prop}\label{S2: Markov semigroup frac}
Given any $\omega\geq 0$, $\Delta-\omega$ generates a symmetric Markov semigroup on $L_2(\M)$. So does $-( -\Delta)^\sigma$.
\end{prop}

For $1\leq p <\infty$, there are two ways to construct the semigroups.
 
(1) First, following a standard process in \cite{Dav89}, for each $1\leq p <\infty$, one can easily show that $ \Delta|_{L_2(\M)\cap L_p(\M)}$ can be extended to the infinitesimal generator of a contraction $C_0$-semigroup on $L_p(\M)$, denoted by $\Delta_p$. 
Following \cite[Chapter~12]{SchiSongVon12} and the same procedure as above, we can define $( -\Delta_p)^\sigma$ and show that $-(  -\Delta_p)^\sigma$ still generates a contraction $C_0$-semigroup on $L_p(\M)$. Moreover, $(  -\Delta_p)^\sigma$ satisfies Balakrishnan's  formula~\eqref{S3: Balakrishnan} as well. 

(2) Second, we can begin with the symmetric Markov semigroup $\{e^{-t(-\Delta)^\sigma}\}_{t\geq 0}$.  Following again the standard process in \cite{Dav89}, one can easily show that the semigroup $\{e^{-t (-\Delta)^\sigma}|_{L_2(\M)\cap L_p(\M)}\}_{t\geq 0}$ can be extended to a contraction $C_0$-semigroup on $L_p(\M)$. We denote its infinitesimal generator by $-(-\Delta)_p^\sigma$.

Since  $\{e^{-t (-\Delta_p)^\sigma} \}_{t\geq 0}$ and $\{e^{-t (-\Delta)_p^\sigma} \}_{t\geq 0}$ coincide on a dense subspace of $L_p(\M)$, we conclude that 
$$
(-\Delta_p)^\sigma=(-\Delta)_p^\sigma.
$$
Therefore, in the sequel, we will  use the notation $(-\Delta_p)^\sigma$ exclusively.
In addition, we always adopt the convention that    $\Delta_2=\Delta$ throughout.
%Following a standard process in \cite{Dav89}, for each $1\leq p <\infty$, one can easily show that $ \Delta|_{L_2(\M)\cap L_p(\M)}$ can be extended to the infinitesimal generator of a contraction $C_0$-semigroup on $L_p(\M)$, denoted by $\Delta_p$ with  the convention that $\Delta_2=\Delta$. 
%Following the same procedure as above, we can define $(  -\Delta_p)^\sigma$ and show that $-(  -\Delta_p)^\sigma$ still generates a contraction $C_0$-semigroup on $L_p(\M)$. Moreover, $(  -\Delta_p)^\sigma$ satisfies Balakrishnan's  formula~\eqref{S3: Balakrishnan} as well. \textcolor{blue}{Check this part carefully.}

Then the following proposition is at our disposal.
\begin{prop}
\label{S2: Lap-frac invert}
For $1\leq p< \infty$, $-(-\Delta_p)^\sigma$ generates a contraction $C_0$-semigroup on $L_p(\M)$, 
and $0\in \rho(\omega+(-\Delta_p)^\sigma)$ for all $\omega>0$.
The semigroup is analytic when $1<p<\infty$.
\end{prop}

%%%%%%%%%%%%%%%%%%%%%%%%%%%%%%%%%%%%%%%%%%%%%%%%%%%%%%%%%%%%%%%%%%%%%%%%
\section{\bf Functional Inequalities  via Subordination}\label{Section 4}

\subsection{\bf A Logarithmic Sobolev  Inequality on  Complete Non-compact Manifolds}\label{Section 4.1}

Assume that  $(\M,g)$ is a  Riemannian manifold satisfying (A1). 
By taking $\Lambda(x)=Mx^{-2/n}$, where $M$ is the constant in (A1), and $V(t)= (\frac{2M}{n}t)^{n/2}$ in \cite[Theorem~1.1]{Grig94}, one can derive that (A1) implies that the heat kernel $p(t,x,y)$ satisfies 
\begin{equation*}
%\label{S2: Gaussian bound on mnfd}
p(t,x,y)\leq C t^{-n/2}.
\end{equation*}
This is equivalent to
\begin{equation}
\label{S2: Ultracontraction mnfd}
\| e^{t\Delta_1} u \|_\infty \leq Ct^{-n/2} \|u\|_1.
\end{equation}
%Denote by $\langle \cdot, \cdot \rangle$ the inner product in $L_2(\M)$.

We start with a theorem by  N.T. Varopoulos, L. Saloff-Coste,  and T. Coulhon for symmetric Markov semigroups.
\begin{theorem}[Theorem~II.5.2 in \cite{VarCosteCoul92}]\label{VSC Thm}
Given a symmetric Markov semigroup $\{e^{tH}\}_{t\geq 0}$ on $L_2(\M)$, when $d>2$, the following conditions are equivalent:
\begin{itemize}
\item[{\em (H)}]  $\| e^{t H} u \|_\infty \leq C t^{-d/2} \|u\|_1$ for $u\in L_1(\M)\cap L_2(\M)$.
\item[{\em (S)}]  $\|u\|^2_{2d/(d-2)} \leq C_1 \langle -H u , u \rangle$ for $u\in D(H)$.
\item[{\em (N)}] $\|u\|^{2+(4/d)}_2 \leq C_2 \|u\|_1^{4/d} \langle -H u , u \rangle$ for $u\in L_1(\M)\cap D(H)$.
\end{itemize}
\end{theorem}
In particular, (H) and (N) are equivalent when $d>0$. 
Following Varopoulos' terminology, the number $d$ is referred to as the dimension of the semigroup  $\{e^{t H}\}_{t\geq 0}$.

Based on \eqref{S2: Ultracontraction mnfd}, we immediately have 
\begin{equation}
\label{S2: Nash ineq Delta}
\|u\|^{2+(4/n)}_2 \leq C_2 \|u\|_1^{4/n} \langle   -\Delta  u , u \rangle,\quad u\in D(\Delta).
\end{equation}
Recall that $\Delta$ generates a symmetric Markov semigroup. Choosing  $B(x)=\frac{1}{C_2}x^{2/n}$  in \cite[Theorem~1]{SchiWang12}, one can derive a Nash type inequality from \eqref{S2: Nash ineq Delta} 
\begin{equation}
\label{S2: Nash ineq frac}
\|u\|^{2+(4\sigma/n)}_2 \leq C_3 \|u\|_1^{4\sigma/n} \langle (-\Delta)^\sigma u, u \rangle = C_3 \|u\|_1^{4\sigma/n} \|( -\Delta)^{\sigma/2} u \|_2^2
\end{equation}
for some $C_3>0$ and all $u\in D((-\Delta)^\sigma)$.
Together with Theorem~\ref{VSC Thm}, this   implies  the following Sobolev-Poincar\'e type inequality
\begin{equation}
\label{S2: Sob ineq frac}
\|u\|^2_{2n/(n-2\sigma)} \leq \widehat{C} \langle ( -\Delta)^\sigma u, u \rangle = \widehat{C} \|( -\Delta)^{\sigma/2} u\|_2^2
\end{equation}
for some $\widehat{C}>0$.
Since $D((-\Delta)^\sigma) \xhookrightarrow{d} D((-\Delta)^{\sigma/2})$,  \eqref{S2: Nash ineq frac} and \eqref{S2: Sob ineq frac} actually  hold for all $u\in D((-\Delta)^{\sigma/2})$.

\eqref{S2: Sob ineq frac} paves the way to the  Logarithmic Sobolev inequality in Theorem~\ref{Thm: functional ineq A1}(vi).
\begin{proof}(of Theorem~\ref{Thm: functional ineq A1}(vi))
Without loss of generality, we assume that $\|u\|_2=1$ so that $u^2 d\mu_g$ is a probability measure. Put $p=2n/(n-2\sigma)$.
By the Jensen's inequality, 
\begin{align*}
\int_\M |u|^2 \ln(|u|)\, d\mu_g =& \frac{1}{p-2}\int_\M  \ln(|u|^{p-2})|u|^2\, d\mu_g \\
 \leq & \frac{1}{p-2}  \ln \|u\|_p^p \\
=& \frac{p}{2(p-2)}  \ln \|u\|_p^2 \\
\leq & \frac{n}{4\sigma}(\ln (\frac{1}{\varepsilon}) +\varepsilon \|u\|^2_{2n/(n-2\sigma)}).
\end{align*}
The last step is due to the fact that $\ln(t)\leq \varepsilon t  - \ln(\varepsilon)$ for all $t,\varepsilon>0$.
%Following the proof for \cite[Theorem~2.4.2]{Dav89}, one has
%$$
%\int_\M |u|^2 \ln(u) \, d\mu_g \leq \frac{n}{4\sigma}(\ln (\frac{1}{\varepsilon}) +\varepsilon \|u\|^2_{2n/(n-2\sigma)}).
%$$
Applying \eqref{S2: Sob ineq frac}, we infer that
\begin{align*}
\int_\M |u|^2 \ln(|u|) \, d\mu_g \leq \frac{n}{4\sigma}(\ln (\frac{1}{\varepsilon}) +\varepsilon \widehat{C} \|( -\Delta)^{\sigma/2} u\|_2^2).
\end{align*}
This establishes the desired inequality.
\end{proof}

Proposition~\ref{S2: Markov semigroup frac} implies that $-(-\Delta)^\sigma$ generates a symmetric Markov semigroup. So we can derive a Strook-Varopoulos type inequality from \cite[Theorem~2.1]{LisSem96}.
\begin{lem}
\label{S2: S-V ineq}
If  $p\in (1,\infty)$ and $u\in D((-\Delta_p)^\sigma)$, then $|u|^{\frac{p-2}{2}}u\in D((-\Delta)^{\sigma/2})$
and
\begin{align*}
\frac{4(p-1)}{p^2}\| ( -\Delta)^{\sigma/2} (|u|^{\frac{p-2}{2}}u)\|_2^2 &\leq   \int_\M |u|^{p-2}u ( -\Delta_p)^\sigma u \, d\mu_g  \\
&\leq   C\| ( -\Delta)^{\sigma/2} (|u|^{\frac{p-2}{2}}u)\|_2^2
\end{align*}
for some $C=C(p).$ 
\end{lem}

A generalization of Strook-Varopoulos inequality can be obtained analogously by means of \cite[Theorem~2.2]{LisSem96}.
\begin{lem}
\label{S2: general S-V ineq}
Let $\psi\in C^2(\R) $ be such that $\psi(s)=0$ for $s\leq 0$, $\psi^\prime(s)>0$ for $s>0$ and $0\leq \psi\leq 1$
and $\displaystyle \sup\limits_{t>0}(1+ \frac{t \psi^{\prime\prime}(t)}{2\psi^\prime(t) })^2<\infty$.
 %and $\phi$ be such that $\phi^\prime=\sqrt{\psi^\prime}$ and $\phi(0)=0$. 
Further, put $G_\psi(t)=t \sqrt{\psi^\prime(t)}$.
If $u\in D((-\Delta_p)^\sigma)\cap L_p(\M,\R_+)$ for some $p\in [1,\infty)$ and $\psi(u)(-\Delta)^\sigma u \in L_1(\M)$, then
$G_\psi(u)\in D((-\Delta)^{\sigma/2})$ and
$$
\|(-\Delta)^{\sigma/2} G_\psi(u)\|_2^2 \leq C \int_\M \psi(u)(-\Delta)^\sigma u \, d\mu_g 
$$
for some $C=C(\psi) $. 
\end{lem}

\begin{remark}
In particular, we can take $\psi$ to be appropriate approximations of the Heaviside function in Lemma~\ref{S2: general S-V ineq}.
\end{remark}

%%%%%%%%%%%%%%%%%%%%%%%%%%%%%%%%%%%%%%%%%%%%%%%%%%%%%%%%%

\subsection{\bf A Logarithmic Sobolev  Inequality on Closed Manifolds}\label{Section 4.2}

Assume that  $(\M,g)$ is a  closed manifold with dimension $n >2$, i.e., it satisfies (A2). For simplicity, we suppose that ${\rm vol}(\M)=1$. 
Here ${\rm vol}(\M)$ is the total volume of $(\M,g)$.
It is well known that $D(\Delta_p)=H^2_p(\M)$ for all $1<p<\infty$, where $H^s_p(\M)$ is the Bessel potential space.
Since it follows from \cite[Theorem 10.3]{AmaHieSim94} that, for certain $c>0$, $c-\Delta_p$ has bounded imaginary power, by \cite[(I.2.9.8)]{Ama95} and \cite[Lemma 2.3.5]{Tana}
$$
D((-\Delta_p))^\sigma \doteq [L_p(\M), H^2_p(\M)]_\sigma \doteq H^{2\sigma}_p(\M),
$$
where $[\cdot,\cdot]_\theta$ is the complex interpolation method. 
Further, it follows from the standard embedding theorem that
\begin{equation}\label{S3.3: Domain embedding}
H^\sigma_2(\M) \hookrightarrow  L_q(\M),
\end{equation}
where $q=\frac{2n}{n-2\sigma}$.

We first start with the well-known  Sobolev-Poinc\'are inequality 
\begin{equation}
\label{S3.3: Poincare-cpt}
\|u-\overline{u}\|_{2^*} \leq C_1 \|\nabla u\|_2,
\end{equation}
where $2^*=\frac{2n}{n-2}$. By the H\"older inequality, we have
$$
\|u-\overline{u}\|_2 \leq \|u-\overline{u}\|_{2^*}^\theta \|u-\overline{u}\|_1^{1-\theta},
$$
where $\theta=\frac{n}{n+2}$. This implies the following Nash inequality 
\begin{equation}
\label{S3.3: Nash-cpt}
\|u-\overline{u}\|_2^{1+2/n} \leq C_1 \|\nabla u\|_2 \|u-\overline{u}\|_1^{2/n}.
\end{equation}
Based on \eqref{S3.3: Nash-cpt}, we will follow the idea in \cite[Proposition~6 and Theorem~1]{SchiWang12} and prove a non-local version of Nash type inequality.
By  \eqref{S3.3: Nash-cpt}, for all $u\in H^2_2(\M)$ with $\|u-\bar{u}\|_1=1$
\begin{align*}
\frac{d}{dt} \|e^{t \Delta} (u - \bar{u})\|_2^2 = &  2 \langle \Delta e^{t \Delta} (u - \bar{u}), e^{t \Delta} (u - \bar{u}) \rangle  
=  - 2 \| \nabla e^{t \Delta} u\|_2^2 \\
\leq &  -2C_1   \|e^{t \Delta} (u - \overline{u})\|_2^{2+4/n}/ \|e^{t \Delta} (u - \overline{u})\|_1^{4/n} \\
\leq & -2C_1   \|e^{t \Delta} (u - \overline{u})\|_2^{2+4/n}.
\end{align*}
By choosing $h(t)= \|e^{t \Delta} (u - \bar{u})\|_2^2= \|e^{t \Delta}  u - \bar{u} \|_2^2$ and $\varphi(t)= 2C_1  t^{1+2/n}$  for $t\geq 0$
in \cite[Lemma~5]{SchiWang12}, we immediately have
$$
\| e^{t \Delta} u - \overline{u}\|_2^2 \leq G^{-1}(G(\|u-\overline{u}\|_2^2)-t),\quad t\geq 0,
$$
holds for all $u\in H^2_2(\M)$ with $\|u-\overline{u}\|_1=1$ and $t \geq 0$. Here 
$$
G(t)=  \frac{n}{4 C_1}(1-t^{-2/n}),\quad t > 0.
$$
\eqref{Phillips formula} implies that for all $u\in H^2_2(\M)$ with $\|u-\overline{u}\|_1=1$ 
\begin{align*}
& \langle  (-\Delta)^\sigma u - \bar{u}, u - \bar{u} \rangle \\
= &  \frac{ \sigma}{\Gamma(1-\sigma)} \int_0^\infty t^{-\sigma-1}  \langle  u-e^{t \Delta} u , u - \bar{u} \rangle \, dt \\
=&  \frac{ \sigma}{\Gamma(1-\sigma)} \int_0^\infty t^{-\sigma-1} \left( \| u-\bar{u} \|_2^2-  \| e^{\frac{t}{2} \Delta} u - \bar{u} \|_2^2  \right) \, dt \\
\geq &  \frac{ \sigma}{\Gamma(1-\sigma)} \int_0^\infty t^{-\sigma-1} \left( \| u-\bar{u} \|_2^2 -G^{-1}(G(\| u-\bar{u} \|_2^2)-\frac{t}{2} ) \right)\, dt =:g(\| u-\bar{u} \|_2^2),
\end{align*}
where
$$
g(r)= \frac{ \sigma}{\Gamma(1-\sigma)} \int_0^\infty t^{-\sigma-1}  \left(r-G^{-1}(G(r)-\frac{t}{2} ) \right)\, dt.
$$
We have
\begin{align}
\notag g(r) = & \frac{ \sigma}{\Gamma(1-\sigma)}\int_0^\infty t^{-\sigma-1}  \left(\int_{G(r)-t/s}^{G(r)} \, dG^{-1}(u)  \right)\, dt \\
\notag =&   \frac{ 1}{2^\sigma \Gamma(1-\sigma)} \int_0^r \frac{ds}{(G(r)-G(s))^\sigma} \\
\label{step}
\geq &  C\int_{r/2}^r \frac{s^{\sigma+2\sigma/n}}{(r-s)^\sigma}\, ds  \\
\notag \geq & C r^{1+\frac{2\sigma}{n}},
\end{align}
where in \eqref{step} we have used \cite[(10)]{SchiWang12} by choosing $B(t)=t^{2/n}$, i.e.
$$
\frac{G(r)-G(u)}{r-u} \geq \frac{u^{-1-2/n}}{2}.
$$
This establishes   the following Nash type inequality
\begin{equation}
\label{S3.3: fractional Nash-cpt}
\|u-\overline{u}\|_2^{1+2\sigma/n} \leq C_2 \| (-\Delta)^{\sigma/2} u \|_2 \| u-\overline{u}\|_1^{2\sigma/n}
\end{equation}
holds for all $u\in D((-\Delta)^\sigma)$ and thus for all $u\in D((-\Delta)^{\sigma/2})$.

%Based on \eqref{S3.3: Nash-cpt}, we can follow the proof of \cite[Proposition~6]{SchiWang12} \textcolor{blue}{reproduce?} and show that
%$$
%\| e^{t \Delta} u - \overline{u}\|_2^2 \leq G^{-1}(G(\|u-\overline{u}\|_2^2)-t),\quad t\geq 0,
%$$
%holds for all $u\in H^2_2(\M)$ with $\|u-\overline{u}\|_1=1$ and $t \geq 0$. Here 
%$$
%G(t)=C_1^{2} \frac{n}{4}(1-t^{-2/n}),\quad t \geq 0.
%$$
%Then one can follow the proof of \cite[Theorem~1]{SchiWang12} step by step and prove that the following Nash type inequality
%\begin{equation}
%\label{S3.3: fractional Nash-cpt}
%\|u-\overline{u}\|_2^{1+2\sigma/n} \leq C_2 \| (-\Delta)^{\sigma/2} u \|_2 \| u-\overline{u}\|_1^{2\sigma/n}
%\end{equation}
%holds for all $u\in D((-\Delta)^\sigma)$ and thus for all $u\in D((-\Delta)^{\sigma/2})$, where $C_2=C_1^\sigma 2^{\frac{n+2\sigma}{2n}}$.

Applying the Young's inequality to \eqref{S3.3: fractional Nash-cpt}, we immediately derive a super Poincar\'e type inequality 
\begin{equation}
\label{S3.3: super poincare}
\|u-\overline{u}\|_2^2 \leq r \|u-\overline{u}\|_1^2 + \beta(r) \| (-\Delta)^{\sigma/2} u \|_2^2, \quad u \in D((-\Delta)^{\sigma/2}) 
\end{equation}
for all $r>0$, where $\beta: (0,\infty)\to (0,\infty)$ is a decreasing function.
In view of the fact ${\rm vol}(\M)=1$, this implies 
\begin{equation}
\label{S3.3: super poincare-2}
\|u-\overline{u}\|_2^2 \leq r \|u-\overline{u}\|_\infty^2 + \beta(r) \| (-\Delta)^{\sigma/2} u \|_2^2, \quad u \in D((-\Delta)^{\sigma/2}).
\end{equation}
A direct computation shows that
$$
\|u-\overline{u}\|_2^2= \|u\|_2^2 - \overline{u}^2
$$
and
$$
\|u-\overline{u}\|_1^2 \leq 4 \|u\|_1^2.
$$
Plugging these results into \eqref{S3.3: super poincare}, we infer that
\begin{equation}\label{S3.3: weak poincare}
\|u\|_2^2 \leq C_3 \|u\|_1^2 + C_4 \| (-\Delta)^{\sigma/2} u \|_2^2 , \quad u \in D((-\Delta)^{\sigma/2}).
\end{equation}
Based on \eqref{S3.3: super poincare-2} and \eqref{S3.3: weak poincare}, \cite[Proposition~1.3]{RockWang01} implies that
\begin{equation*}
\| u -\overline{u}\|_2^2 \leq C \| (-\Delta)^{\sigma/2} u \|_2^2 , \quad u \in D((-\Delta)^{\sigma/2}).
\end{equation*}
Combining with \eqref{S3.3: Domain embedding}, we establish the following  Sobolev-Poincar\'e type inequality:
\begin{equation}\label{S3.3: fractional Poincare-cpt}
\| u -\overline{u}\|_{2n/(n-2\sigma)}^2 \leq \widetilde{C} \| (-\Delta)^{\sigma/2} u \|_2^2 , \quad u \in D((-\Delta)^{\sigma/2}).
\end{equation}
Note that by the H\"older  inequality, \eqref{S3.3: fractional Poincare-cpt} implies \eqref{S3.3: fractional Nash-cpt}. 

Based on \eqref{S3.3: fractional Poincare-cpt}, Theorem~\ref{Thm: functional ineq A2}(vi) immediately follows from a similar   proof of Theorem~\ref{Thm: functional ineq A1}(vi).

Finally, we would like to point out that  Lemmas~\ref{S2: S-V ineq} and \ref{S2: general S-V ineq} still hold true for  closed manifolds $(\M,g)$.

%We thus obtain the following proposition.

%%%%%%%%%%%%%%%%%%%%%%%%%%%%%%%%%%%%%%%%%%%

\subsection{\bf Other Functional Inequalities, Heat Kernel and Semigroup Estimates via Subordination}\label{Section 4.3}

In this subsection, we will continue the discussion in Sections~\ref{Section 4.1} and \ref{Section 4.2} and derive the remaining functional inequalities  and the heat kernel and semigroup estimates for $\{e^{-t (-\Delta)^\sigma}\}_{t\geq 0}$ in Theorems~\ref{Thm: functional ineq A1} and \ref{Thm: functional ineq A2}.

First, we consider a  Riemannian manifold $(\M,g)$ satisfying (A1) with  $n >2$. 
Applying the Young's inequality to the Nash type inequality~\eqref{S2: Nash ineq frac}, we obtain a super Poincar\'e type inequality 
\begin{equation}
\label{A: frac super poincare}
\|u\|_2^2 \leq r \|u \|_1^2 + \beta(r) \| (-\Delta)^{\sigma/2} u \|_2^2, \quad u \in D((-\Delta)^{\sigma/2}),
\end{equation}
where $\beta: (0,\infty)\to (0,\infty)$ is a decreasing function.
The Sobolev-Poincar\'e type inequality \eqref{S2: Sob ineq frac}  and Theorem~\ref{VSC Thm} imply that 
\begin{equation}
\label{A: frac heat semigroup}
\| e^{-t (-\Delta)^\sigma} u \|_\infty \leq C t^{-n/2\sigma} \|u\|_1
\end{equation} 
for all $u\in L_1(\M)$.

We denote by $p_\sigma(x,y,t)$ the heat kernel of the semigroup $\{e^{-t(-\Delta)^\sigma}\}_{t\geq 0}$.
Given any $\Omega\subset\subset \M$, it follows from Proposition~\ref{S2: Markov semigroup frac} that
$$
\int_\Omega p_\sigma(x,y,t)\, d\mu_g(y) \leq 1 .
$$
Letting $\Omega$ invade $\M$ yields
$$
\int_\M p_\sigma(x,y,t)\, d\mu_g(y) \leq 1.
$$
It is evident that $p_\sigma(x,y,t)\geq 0$ for all $x,y\in \M$ and $t>0$ and symmetric in $x$ and $y $. 
This, in particular, implies that for every fixed $x$, the heat kernel $p_\sigma(x,y,t)$, as a function of $y$, has $L_1-$norm no larger than $1$.
Using the semigroup property
$$
\int_\M p_\sigma(x,y,t) p_\sigma(y,z,s)\, d\mu_g(y)  = p_\sigma(x,z,t+s) 
$$
and \eqref{A: frac heat semigroup},
we can derive the heat kernel upper bound
\begin{equation}
\label{A: frac heat kernel upper bound}
p_\sigma(x,y,t) \leq C t^{-n/2\sigma},\quad x,y\in \M,\, t>0.
\end{equation} 
By \cite[Theorem~3.1]{Grig97}, we can further derive the Gaussian upper bound for the heat kernel
\begin{equation}
\label{A: frac heat kernel Gauss upper bound}
p_\sigma(x,y,t) \leq C t^{-n/2\sigma} e^{-d^2(x,y)/At},\quad x,y\in \M,\, t>0,
\end{equation} 
for some $C,A>0$. Here $d(x,y)$ is the distance between $x$ and $y$.
By Jensen's inequality, for all $1\leq p<\infty$ and $t>0$
\begin{align*}
|e^{-t (-\Delta)^\sigma} u (x)|^p =& |\int_M p_\sigma(x,y,t) u(y)\,d\mu_g(y)|^p \\
\leq &   \int_M p_\sigma(x,y,t) |u(y)|^p \,d\mu_g(y)  \\
\leq & C  t^{-n/2\sigma}e^{-d^2(x,y)/At} \|u\|_p^p.
\end{align*}
This implies that the semigroup $\{e^{-t (-\Delta)^\sigma} \}_{t\geq 0}$ is ultracontractive, i.e.
\begin{equation}
\label{A: frac heat semigroup ultracontractive}
\|e^{-t (-\Delta)^\sigma} u \|_\infty \leq C t^{-\frac{n}{2 p\sigma }} e^{\frac{-d^2(x,y)}{Apt}}\|u\|_p.
\end{equation}
Finally, \eqref{S2: Nash ineq frac}, \eqref{S2: Sob ineq frac}, \eqref{A: frac super poincare}, \eqref{A: frac heat kernel Gauss upper bound} and \eqref{A: frac heat semigroup ultracontractive} give Theorems~\ref{Thm: functional ineq A1}.

\begin{remark}
The estimate \eqref{A: frac heat semigroup} can also be derived by using \eqref{S2: def subordinate semigroup}. 
\end{remark}

Now we turn our attention to a closed Riemannian manifold $(\M,g)$.
%By \cite[(6.40)]{Grig99},  the heat kernel  $p(x,y,t)$ of the semigroup $\{e^{t\Delta}\}_{t\geq 0}$ fulfils
%\begin{equation}
%\label{A: heat kernel closed}
%p(x,y,t) \leq C \max\{1, t^{-n/2}\} e^{-d^2(x,y)/At}
%\end{equation}
%for some $C,A>0$.
%Note that \cite[(6.40)]{Grig99} is satisfied for manifolds of bounded geometry and thus for closed manifolds.
%Based on \eqref{A: heat kernel closed}, we can derive that 
%\begin{equation*}
%\label{A: heat semigroup closed}
%\| e^{t  \Delta} u \|_\infty \leq   C \max\{1, t^{-n/2}\} \|u\|_1.
%\end{equation*}

Pick any $u\in H^{2\sigma}_2(\M)$ with  $\|u-\overline{u}\|_1=1$. Then
\begin{align*}
\frac{d}{dt}\| e^{-t(-\Delta)^\sigma} (u-\overline{u})\|_2^2 \leq & -2 \langle (-\Delta)^\sigma e^{-t(-\Delta)^\sigma} (u-\overline{u}), e^{-t(-\Delta)^\sigma} (u-\overline{u}) \rangle \\
\leq & -C \|e^{-t(-\Delta)^\sigma} (u-\overline{u})\|_2^{2+4\sigma/n} \|e^{-t(-\Delta)^\sigma} (u-\overline{u})\|_1^{-4\sigma/n}\\
\leq & -C \|e^{-t(-\Delta)^\sigma} (u-\overline{u})\|_2^{2+4\sigma/n}.
\end{align*}
The second line follows from \eqref{S3.3: fractional Nash-cpt} and the third is a direct consequence of the contraction of the semigroup $\{e^{-t(-\Delta)^\sigma}\}_{t\geq 0}$.
This implies that for any $u\in H^{2\sigma}_2(\M)$
$$
\|e^{-t(-\Delta)^\sigma} (u-\overline{u})\|_2 \leq C t^{-\frac{n}{4\sigma}} \|u-\overline{u}\|_1.
$$
By the triangle inequality, we immediately have
$$
\|e^{-t(-\Delta)^\sigma} u\|_2 \leq  C (t^{-\frac{n}{4\sigma}} +1) \|u\|_1 \leq C \max\{1, t^{-\frac{n}{4\sigma}}\}\|u\|_1.
$$
Given any $f\in L_1(\M)$, 
\begin{align*}
\langle e^{-t(-\Delta)^\sigma} u, f \rangle \leq & \|e^{-t(-\Delta)^\sigma/2} u \|_2 \| e^{-t(-\Delta)^\sigma/2} f \|_2 \\
\leq &  C \max\{1, t^{-\frac{n}{2\sigma}}\}\|u\|_1 \|f\|_1,
\end{align*}
which implies
\begin{equation*}
\label{A: frac heat semigroup closed}
\| e^{-t (-\Delta)^\sigma} u \|_\infty \leq C \max\{1, t^{-n/2\sigma}\} \|u\|_1.
\end{equation*}
Now following the   argument leading to \eqref{A: frac heat kernel upper bound}, we can derive the Gaussian upper bound for the heat kernel
\begin{equation}
\label{A: frac heat kernel upper bound-closed}
p_\sigma(x,y,t) \leq C \max\{1,t^{-n/2\sigma}\}  e^{-d^2(x,y)/At},\quad x,y\in \M,\, t>0,
\end{equation} 
and the ultracontractivity
\begin{equation}
\label{A: frac heat semigroup ultracontractive-closed}
\|e^{-t (-\Delta)^\sigma} u \|_\infty \leq C \max\{1, t^{-\frac{n}{2 p\sigma }} \} e^{\frac{-d^2(x,y)}{Apt}}\|u\|_p, \quad 1\leq p<\infty.
\end{equation} 
In sum, \eqref{S3.3: fractional Nash-cpt}, \eqref{S3.3: super poincare}, \eqref{S3.3: fractional Poincare-cpt}, \eqref{A: frac heat kernel upper bound-closed} and \eqref{A: frac heat semigroup ultracontractive-closed} give Theorems~\ref{Thm: functional ineq A2}.

%%%%%%%%%%%%%%%%%%%%%%%%%%%%%%%%%%%%%%%%%%%%%%%%%%%%%%%%%

\section{\bf Solutions to the Fractional Porous Medium Equation}
\label{Section 5}

%%%%%%%%%%%%%%%%%%%%%%%%%%%%%%%%%%%%%%%

To prove the global well-posedness of \eqref{S1: FPME}, we first study the following generalization of \eqref{S1: FPME} with $\omega\geq 0$
\begin{equation}
\label{S4-FPME-omega}
\left\{\begin{aligned}
\partial_t u +[\omega+(-\Delta)^\sigma] (|u|^{m-1}u  )&=0   &&\text{on}&&\M\times (0,\infty);\\
u(0)&=u_0    &&\text{on}&&\M .
\end{aligned}\right.
\end{equation}
In \cite{RoidosShao18}, we  established the existence and uniqueness of a strong solution  to \eqref{S1: FPME} on an incomplete Riemannian manifold with conical singularities and finite volume. 
We will nevertheless state a brief proof for  the existence and uniqueness part  for two reasons: 
(1) we will adopt a more elegant argument which is applicable to   manifolds with infinite volume;
%(2) the new approach is simpler and generalize the results in \cite{RoidosShao18}; 
(2) the proofs of the asymptotic behaviors of  solutions rely on how they are constructed. 
%The proofs for the cases ${\rm vol}(\M)=\infty$ and ${\rm vol}(\M)<\infty$ are essentially the same.
%For this, in this section, we will focus on the former.
% and only state the necessary changes for the finite volume case.

In this section, we assume that the initial datum $u_0\in L_1(\M)\cap L_{m+1}(\M)$. 
The initial condition will be relaxed to $u_0\in L_1(\M)\cap L_2(\M)$ when $m>1$ and the underlying manifolds satisfying (A1) or (A2) in the next two sections.

\subsection{\bf Definition of  Solutions}\label{Section 5.1}
 
Let $\Phi(x)=|x|^{m-1}x$  and $\beta=\Phi^{-1}$. Note that $\Phi$ and $\beta$ are maximal monotone graphs in $\R^2$ containing $(0,0)$.
We define the notions of solutions to \eqref{S4-FPME-omega} as follows.
\begin{definition}\label{Def: solution}
Given $\omega\geq 0$, we say that $u$ is a weak solution to \eqref{S4-FPME-omega}   if
\begin{itemize}
\item $u\in L_{\infty,loc}((0,\infty), L_{m+1}(\M))$, and 
\item $ (-\Delta)^{\sigma/2} \Phi(u) , \sqrt{\omega}\Phi(u)\in L_{2,loc}((0,\infty), L_2(\M) )$, and
\item $u\in C([0,\infty),L_1(\M))$.
\end{itemize}
Moreover, for every $\phi\in C^1_c([0,\infty)\times \M)$, it holds that
\begin{align}
\label{S4: weak sol} 
 \notag &\int_0^\infty \int_\M ( -\Delta)^{\sigma/2} \Phi(u)  ( -\Delta)^{\sigma/2}  \phi \, d\mu_g dt +\omega \int_0^\infty \int_\M   \Phi(u)    \phi \, d\mu_g dt\\
 = &  \int^\infty_0\int_\M u \partial_t \phi \, d\mu_g dt 
 +\int_\M  u_0\phi(0)\, d\mu_g.
\end{align}
If, in addition, $u$ satisfies
\begin{itemize}
\item when $m=1$,   $\partial_t u, (-\Delta)^\sigma \Phi(u) \in L_{2,loc}((0,\infty),   L_2(\M))$ and  further
$u\in C([0,\infty),L_2(\M))$; or
\item when $m\in (0,1)\cup (1,\infty)$, $\partial_t u, (-\Delta)^\sigma \Phi(u) \in L_{\infty,loc}((0,\infty), L_1(\M) ),$
\end{itemize}
we call $u$ a strong solution to \eqref{S4-FPME-omega}.
%and $lim_{t\to 0} u (t, \cdot)=u_0$ in $\cH^{0,\gamma_1}_1(\M)$.
\end{definition}

%\begin{remark}
%The definitions of weak and strong solutions in Definition~\ref{Def: solution} are slightly weaker than those in \cite[Definitions~5.4 and 7.1]{RoidosShao18}. 
%However, for $u_0\in L_1(\M)\cap L_{m+1}(\M)$, a weak  solution to \eqref{S1: FPME}   indeed satisfies the conditions in \cite[Definition~5.4]{RoidosShao18}; if further $m>1$, the solution is strong in the sense of \cite[Definition~7.1]{RoidosShao18}.
%The weaker version  Definition~\ref{Def: solution}  is only needed to generalize the results in \cite{RoidosShao18}. 
%\end{remark}

\subsection{\bf Existence of Weak Solution}\label{Section 5.2}

By Proposition~\ref{Prop: m-accretive}, the operator 
$$
\cA(u):=[\omega+(-\Delta_1)^\sigma]\Phi(u)  :D(\cA)\subset L_1(\M)\to L_1(\M)
$$ 
is $m$-accretive and with dense domain.
We can apply the Crandall-Liggett generation theorem \cite[Theorem~I]{CraLig71} and  prove the existence of a  global mild solution to \eqref{S4-FPME-omega}.
%Mild solutions are defined as the limit of a sequence of approximation solutions by implicit time discretization. 
More precisely, given $T>0$,
for a partition $\mathcal{P}=\{0=t_0<t_1 <\cdots < t_n=T\}$ of $[0,T)$ with $\Delta T_k = t_k-t_{k-1}$, the discretized problem to \eqref{S4-FPME-omega} is 
\begin{equation}
\label{S4: DPME}
\Delta T_k [\omega+(-\Delta)^\sigma]\Phi(u_{n,k;\omega})  = u_{n,k-1;\omega}- u_{n,k;\omega}  \quad \text{with} \quad 
u_{n,0;\omega} = u_0  .
\end{equation}
For simplicity, we may take $\Delta T_k=T/n$.
The piecewise solution is defined as
$$
u_{n;\omega}(0)=u_0,\quad u_{n;\omega}(t)= u_{n,k;\omega} \quad \text{for } t\in (t_{k-1},t_k].
$$
The the uniform limit 
$
u_\omega\in C([0,T], L_1(\M))
$
of $u_{n;\omega}$, i.e. for any $\varepsilon>0$, 
\begin{equation}
\label{S5: mild sol def and converg}
\|u_\omega(t)-u_{n;\omega}(t)\|_1<\varepsilon, \quad t\in [0,T]
\end{equation}
for sufficiently large $n$, is the unique global mild solution to \eqref{S4-FPME-omega}.

$u_\omega$ is $L_q-$contractive for all $1\leq q\leq m+1$. Indeed, 
\cite[Proposition~4]{BreStr73} implies that  
\begin{equation}
\label{S4: L_infty contraction dis sol}
\|u_{n,k;\omega}\|_q  \leq 	\|u_{n,k-1;\omega}\|_q \leq \|u_0\|_q\quad 1\leq q \leq m+1,
\end{equation}
and it follows from  Fatou's Lemma and \eqref{S5: mild sol def and converg} that for any $0\leq t$,
\begin{equation}
\label{S4: unif est 2}
 \|u_\omega(t)\|_q\leq \|u_0\|_q.
\end{equation}
In view of \eqref{S4: DPME}, \eqref{S4: L_infty contraction dis sol} reveals that $[\omega+(-\Delta)^\sigma ]\Phi(u_{n,k;\omega}) \in L_1(\M)\cap  L_{m+1}(\M)$. 
 
Multiplying  \eqref{S4: DPME} by $\Phi(u_{n,k;\omega})$ and integrating over $\M$ give 
\begin{align}
\label{S4: weak est 1}
\notag& \frac{T}{n}  \int_\M \Big[ | ( -\Delta)^{\sigma/2}) \Phi(u_{n,k;\omega})|^2 +\omega |\Phi(u_{n,k;\omega})|^2 \Big]\, d\mu_g  \\
\notag =& \int_\M  u_{n,k-1;\omega} \Phi(u_{n,k;\omega})\, d\mu_g -\int_\M  u_{n,k;\omega} \Phi(u_{n,k;\omega})\, d\mu_g \\
\leq & \frac{1}{m+1} (\int_\M | u_{n,k-1;\omega}|^{m+1} \, d\mu_g  - \int_\M | u_{n,k;\omega}|^{m+1} \, d\mu_g ).
\end{align}
We have used the H\"older and Young's inequalities in \eqref{S4: weak est 1}.
Summing over   $k=1,2,\cdots,n$  yields
$$
\int_0^T \int_\M \Big[ | ( -\Delta)^{\sigma/2}) \Phi(u_{n ;\omega})|^2 +\omega |\Phi(u_{n ;\omega})|^2 \Big]\, d\mu_g dt \leq \frac{1}{m+1}\int_\M |u_0 |^{m+1}\, d\mu_g ;
$$
and thus 
\begin{equation}
\label{S4: unif est 3}
%\omega \|u_{ \omega}\|_{L_2((0,T), L_{2m}(\M))} + 
\|(-\Delta)^{\sigma/2} \Phi(u_{ \omega})\|_{L_2((0,T), L_2(\M))} 
\leq  \frac{1}{m+1}\int_\M |u_0 |^{m+1}\, d\mu_g .
\end{equation}
Multiplying  \eqref{S4: DPME} by $\phi\in C^1_c([0,\infty)\times\M)$ and integrating over $\M$ yield
\begin{align*}
& \int_\M  ( -\Delta)^{\sigma/2}  \Phi(u_{n,k;\omega}) ( -\Delta)^{\sigma/2}  \phi \, d\mu_g + \omega\int_\M    \Phi(u_{n,k;\omega})  \phi \, d\mu_g \\
=& \frac{n}{T}\int_\M (u_{n,k-1;\omega}-u_{n,k;\omega})\phi\, d\mu_g.
\end{align*}
Then integrate over $[t_{k-1},t_k)$ and sum over $k=1,2,\cdots,n$. The right hand side equals
\begin{align*}
&\frac{n}{T} \sum\limits_{k=1}^n \int_{t_{k-1}}^{t_k} \int_\M (u_{n,k-1;\omega} -u_{n,k;\omega})\phi \, d\mu_g dt\\
=& \int_0^T\int_\M u_{n;\omega}(t) \frac{\phi(t+T/n) -\phi(t)}{T/n}\, d\mu_g dt +  \frac{n}{T}\int_0^{t_1} \int_\M  u_0 \phi(t)\, d\mu_g dt\\
& - \frac{n}{T}\int_{t_{n-1}}^T \int_\M u_{n;\omega}(T) \phi(t+T/n)\, d\mu_g dt.
\end{align*}
Pushing $n\to\infty$ yields 
\begin{align}
\label{S4: weak sol-u_omega}
\notag & \int^T_0\int_\M u_\omega \partial_t \phi \, d\mu_g dt+\int_\M  u_0 \phi(0)\, d\mu_g -  \int_\M  u_\omega(T) \phi(T)\, d\mu_g\\
=&\int_0^T \int_\M ( -\Delta)^{\sigma/2} \Phi(u_\omega)  ( -\Delta)^{\sigma/2}  \phi \, d\mu_g dt +  \omega \int_0^T \int_\M\Phi(u_\omega)    \phi \, d\mu_g dt. 
\end{align}

%%%%%%%%%%%%%%%%%%%%%%%%%%%%%%%%%%%%%%

Take any positive   sequence $ \omega_k \to 0^+$. We rewrite \eqref{S4: DPME} as 
\begin{equation*}
\left\{\begin{aligned}
u_{n,k;\omega_h} + \Delta T_k [\omega_l+(-\Delta)^\sigma]\Phi(u_{n,k;\omega_h}) &= u_{n,k-1;\omega_h}-\Delta T_k (\omega_h-\omega_l)   \Phi(u_{n,k;\omega_h}) ;\\
u_{n,0;\omega_h} &=u_0  
\end{aligned}\right.
\end{equation*}
for $h<l$.
The existence of $u_{n;\omega_h}$ has already been established. Now we try to estimate $\|u_{n,k;\omega_h}-u_{n,k;\omega_l}\|_1$. 
Proposition~\ref{S4: semilinear thm} implies
\begin{align*}
\|u_{n,1;\omega_h}-u_{n,1;\omega_l}\|_1 \leq &\frac{T}{n} (\omega_h-\omega_l) \|  \Phi(u_{n,1;\omega_h})\|_1= \frac{T}{n} (\omega_h-\omega_l)\|u_{n,1;\omega_h}\|_m^m  \\
\leq & \frac{T}{n} (\omega_h-\omega_l)\|u_0\|_m^m,
\end{align*}
where the last step is due to \eqref{S4: L_infty contraction dis sol};
and
\begin{align*}
\|u_{n,2;\omega_h}-u_{n,2;\omega_l}\|_1  \leq & \|u_{n,1;\omega_h}-u_{n,1;\omega_l}\|_1 +  \frac{T}{n} (\omega_h-\omega_l) \|  \Phi(u_{n,2;\omega_h})\|_1 \\
\leq &\frac{2T}{n} (\omega_h-\omega_l)\|u_0\|_m^m.
\end{align*}
By induction, we thus have
$$
\|u_{n,k;\omega_h}-u_{n,k;\omega_l}\|_1 \leq \frac{kT}{n} (\omega_h-\omega_l)\|u_0\|_m^m.
$$
This implies that 
$$
\|u_{\omega_h}-u_{\omega_l}\|_{C([0,T], L_1(\M))} \leq (\omega_h-\omega_l)\|u_0\|_m^m.
$$
We conclude that $(u_{\omega_k})_k$ is Cauchy in $C([0,T], L_1(\M))$ and thus converges to some $u\in C([0,T], L_1(\M))$.
\eqref{S4: unif est 2} and \eqref{S4: unif est 3}  imply that  
\begin{align*}
%\label{S4:converg-1}
(-\Delta)^{\sigma/2}\Phi(u_\omega) & \rightharpoonup (-\Delta)^{\sigma/2}\Phi(u) \quad && \text{in}\quad L_2((0,T), L_2(\M))\\
%\label{S4:converg-2}
u_\omega  & \rightharpoonup u   \quad && \text{in}\quad L_\infty((0,T),L_{m+1} (\M)) 
\end{align*}
as $\omega\to 0^+$.
In view of \eqref{S4: unif est 2}, pushing $\omega \to 0^+$ in \eqref{S4: weak sol-u_omega} yields that 
\begin{align*}
  & \int^T_0\int_\M u \partial_t \phi \, d\mu_g dt+\int_\M u_0\phi(0)\, d\mu_g - \int_\M u(T)\phi(T)\, d\mu_g \\
     = &   \int_0^T \int_\M   ( -\Delta)^{\sigma/2 } \Phi(u  )  ( -\Delta)^{\sigma/2 }  \phi \, d\mu_g dt  
\end{align*}
for any $\phi\in C^1_c([0,\infty)\times \M)$.
Note that the estimate~\eqref{S4: unif est 3} holds for all $T>0$. We thus have
\begin{align*}
    \int^\infty_0\int_\M u \partial_t \phi \, d\mu_g dt =&  \lim\limits_{T\to \infty} \int^T_0\int_\M u \partial_t \phi \, d\mu_g dt  \\ 
     = &   \lim\limits_{T\to \infty} \int_0^T \int_\M   ( -\Delta)^{\sigma/2 } \Phi(u  )  ( -\Delta)^{\sigma/2 }  \phi \, d\mu_g dt\\  &-\int_\M u_0\phi(0)\, d\mu_g +  \lim\limits_{T\to \infty} \int_\M u(T)\phi(T)\, d\mu_g \\
     =& \int_0^\infty \int_\M   ( -\Delta)^{\sigma/2 } \Phi(u  )  ( -\Delta)^{\sigma/2 }  \phi \, d\mu_g dt -\int_\M u_0\phi(0)\, d\mu_g.
\end{align*}
Therefore, $u$ is a weak solution to \eqref{S1: FPME}.

%%%%%%%%%%%%%%%%%%%%%%%%%%%%%%%%%%%%%%%%%%%%%%%%%%%%%%
\subsection{\bf Existence and Uniqueness of Strong Solution}\label{Section 5.3}

The  strategies in this subsection are picked from \cite[Sections~6 and 8]{PalRodVaz12}. However, we will generalize some results in \cite{PalRodVaz12} to \eqref{S4-FPME-omega} with $\omega>0$, which will be used in the next two sections.

\begin{prop}\label{Prop: exist strong sol}
For any $\omega\geq 0$, the weak solutions $u_\omega$ to \eqref{S4-FPME-omega} constructed in Section~\ref{Section 5.2} are strong solutions.
\end{prop}
\begin{proof}
When $m=1$, the standard semigroup theory, c.f. \cite[Theorem 4.1.4]{Pazy}, and Proposition~\ref{S2: Lap-frac invert} imply that  \eqref{S1: FPME} has  a unique solution in the class
\begin{align*}
\tilde{u}_\omega\in   C^1([0,\infty),L_2(\M))\cap C([0,\infty),D((-\Delta)^\sigma)).
\end{align*}
Define
$$
\phi(t)=\int_t^T ( u_\omega - \tilde{u}_\omega )\, ds ,\quad 0\leq t\leq T, 
$$
and $\phi\equiv 0$ for $t\geq T$,
which belongs to $H^1_2((0,T), D((-\Delta)^{\sigma/2}))\hookrightarrow C([0,T], L_2(\M))$. By a standard approximation argument, $\phi$ is a valid   test function in \eqref{S4: weak sol}. We have
\begin{align*}
\int_0^T\int_\M &(-\Delta)^{\sigma/2} ( u_\omega - \tilde{u}_\omega )(t) [\int_t^T (-\Delta)^{\sigma/2} ( u_\omega - \tilde{u}_\omega )(s)\, ds] \, d\mu_g dt \\
& + \omega \int_0^T\int_\M   ( u_\omega - \tilde{u}_\omega )(t) [\int_t^T  ( u_\omega - \tilde{u}_\omega )(s)\, ds] \, d\mu_g dt \\
&+ \int_0^T\int_\M (u_\omega-\tilde{u}_\omega)^2(t)   \, d\mu_g dt=0,
\end{align*}
which is equivalent to
\begin{align*}
& \frac{1}{2}\int_\M \Big[ \int_0^T   (\Delta)^{\sigma/2} ( u_\omega - \tilde{u}_\omega )(t)\, dt \Big]^2 \, d\mu_g  \\
 + & \frac{\omega}{2} \int_\M \Big[ \int_0^T  ( u_\omega - \tilde{u}_\omega )(t)\, dt \Big]^2 \, d\mu_g   
  + \int_0^T\int_\M (u_\omega-\tilde{u}_\omega)^2(t)    \, d\mu_g dt=0.
\end{align*}
All integrals need to be zero. We thus infer that $u_\omega=\tilde{u}_\omega$. 
This implies that
$$
\partial_t u_\omega \in L_{2,loc}((0,\infty), L_2(\M)) \quad \text{and} \quad u_\omega\in C ([0,\infty),L_2(\M)).
$$
Therefore, $u_\omega$ is   a strong solution to \eqref{S4-FPME-omega}.

When $m\neq 1$, the argument follows the idea in  \cite[Lemma~8.1 and Theorem~8.2]{PalRodVaz12}. 
For any $f\in L_{1,loc}(0,T)$, define the Steklov average of $f$ by
$$
f^h(t):=\frac{1}{h}\int^{t+h}_t f(s) \, ds,
$$
and 
$$
\delta^h f (t):=\partial_t f^h(t)= \frac{f(t+h)-f(t)}{h}\quad \text{a.e.}
$$
The weak formulation \eqref{S4: weak sol} can be restated as
\begin{align}
\label{S3: weak sol2}
\notag & \int_0^T \int_\M (\delta^h u_\omega )\phi \, d\mu_g\,dt  + \omega \int_0^T \int_\M (\Phi(u_\omega))^h  \phi \, d\mu_g\,dt
\\
&+ \int_0^T \int_\M  (-\Delta)^{\sigma/2} (\Phi(u_\omega))^h (-\Delta)^{\sigma/2} \phi \, d\mu_g\,dt =0.
\end{align}
For any $[\tau,S]\subset (0,T)$, we choose $\zeta\in C_0^1((0,T),[0,1])$ such that $\zeta\equiv 1$ on $[\tau,S]$ and vanishes outside $[\tau^\prime, S^\prime]$ for some $[\tau^\prime, S^\prime]\subset (0,T)$ with $[\tau,S]\subset (\tau^\prime, S^\prime)$.
Let us take $\phi=\zeta \delta^h (\Phi(u_\omega))$.
% for $0\leq \zeta\in C^1([0,T])\cap C_0((0,T))$. 
Then \eqref{S3: weak sol2} yields
\begin{align}
\label{S3: weak sol3}
\notag &\quad\int_0^T \int_\M \zeta (\delta^h u_\omega) \delta^h (\Phi(u_\omega)) \, d\mu_g\,dt + \omega \int_0^T \int_\M \zeta   (\Phi(u_\omega))^h  \partial_t(\Phi(u_\omega))^h \, d\mu_g\,dt \\
&+ \int_0^T \int_\M \zeta  (-\Delta)^{\sigma/2} (\Phi(u_\omega))^h (-\Delta_g)^{\sigma/2} \partial_t(\Phi(u_\omega))^h \, d\mu_g\,dt
=0.
\end{align}
Since  $(\delta^h u_\omega)(\delta^h \Phi(u_\omega))\geq c (\delta^h (|u_\omega|^{(m-1)/2}u_\omega))^2$, cf. \cite[Section~5.3]{PalRodVaz11},
the first term on the left hand side of \eqref{S3: weak sol3} satisfies that
$$
\int_0^T \int_\M \zeta (\delta^h u_\omega ) \delta^h (\Phi(u_\omega)) \, d\mu_g\,dt \geq c \int_0^T \int_\M \zeta (\delta^h (|u_\omega|^{(m-1)/2}u_\omega))^2 \, d\mu_g\,dt.
$$
The  second  and third terms on the left hand side of \eqref{S3: weak sol3} can be estimated as follows 
\begin{align*}
\notag   \Big|\int_0^T \int_\M \zeta    (\Phi(u_\omega))^h   \partial_t(\Phi(u_\omega))^h \, d\mu_g\,dt \Big|  
\leq &  C\int_{\tau^\prime}^{S^\prime} \int_\M |\zeta^\prime | [   (\Phi(u_\omega))^h ]^2\, d\mu_g\,dt\\
\leq & C \int_{\tau^\prime}^{S^\prime} \int_\M   [  (\Phi(u_\omega))^h ]^2\, d\mu_g\,dt  
\end{align*}
and similarly
\begin{align*}
\notag & \Big|\int_0^T \int_\M \zeta  (-\Delta)^{\sigma/2} (\Phi(u_\omega))^h (-\Delta)^{\sigma/2} \partial_t(\Phi(u_\omega))^h \, d\mu_g\,dt \Big| \\
\leq & C \int_{\tau^\prime}^{S^\prime} \int_\M   [ (-\Delta)^{\sigma/2} (\Phi(u_\omega))^h ]^2\, d\mu_g\,dt  .
\end{align*}
It follows from \eqref{S4: unif est 3} that
$$
\int_\tau^S \int_\M  (\delta^h (|u_\omega|^{(m-1)/2}u_\omega))^2 \, d\mu_g\,dt \leq \int_0^T \int_\M \zeta (\delta^h (|u_\omega|^{(m-1)/2}u_\omega))^2 \, d\mu_g\,dt \leq C
$$
for some $C>0$ independent of $h$,
and thus
$$
\partial_t (|u_\omega|^{(m-1)/2}u_\omega)\in L_{2,loc}((0,T), L_2(\M) ).
$$
Since \cite[Theorems 1 and 2]{BenCra81} implies $u_\omega\in BV((\tau,T), L_1(\M))$ for any $\tau>0$, it then follows from \cite[Theorem~1.1]{BenGar95} that
%Now one can apply the same argument for \cite[Lemma~8.2 and Corollary~8.3]{Vaz07} and prove that 
$$
\partial_t u_\omega \in L_{\infty,loc}((0,T), L_1(\M)) 
$$
with 
$$
\|\partial_t u_\omega(t))\|_1 \leq \frac{2}{|m-1| t}\|u_0\|_1.
$$
This leads to 
$$
  (-\Delta_g)^\sigma \Phi(u_\omega) \in L_{\infty,loc}((0,T),L_1(\M)).
$$
\end{proof}

The uniqueness of the strong solution follows from the following lemma.
\begin{lem}\label{lem: unique strong sol}
Given $u_0\in L_1(\M)$ when $m\neq 1$ or $u_0\in L_2(\M)$ when $m= 1$, \eqref{S1: FPME} has at most one strong solution.
\end{lem}
\begin{proof}
When $m\neq 1$, if $u_1$, $u_2$, are strong solutions to \eqref{S1: FPME} with
initial data $u_{0,1}, u_{0,2} \in L_1(\M)$, then it follows from Lemma~\ref{S2: general S-V ineq} and the proof of \cite[Theorem~6.2]{PalRodVaz12} that  
for every $0\leq t_1<t_2$ it holds
$$
\int_\M (u_1-u_2)_+ (t_2)\, d\mu_g \leq \int_\M (u_1-u_2)_+ (t_1)\, d\mu_g.
$$
When $m=1$, if $u_1$, $u_2$, are strong solutions to \eqref{S1: FPME} with
initial data $u_{0,1}, u_{0,2} \in L_2(\M)$, then  we have
$$
\int_\M \partial_t (u_1-u_2) ( u_1-u_2)\, d\mu_g = - \int_\M | (-\Delta)^{\sigma/2} (u_1-u_2)|^2 \, d\mu_g \leq 0,
$$
which implies that for every $0< t_1<t_2$
$$
\|(u_1-u_2)(t_2)\|_2 \leq  \|(u_1-u_2)(t_1)\|_2.
$$
The fact that $u_1,u_2\in C([0,\infty),L_2(\M))$ then implies the uniqueness of strong solution.
\end{proof}

\subsection{\bf Proof of Theorem~\ref{S4: global weak sol}}\label{Section 5.4}

\begin{proof}
(of Theorem~\ref{S4: global weak sol})
We have already proved the existence and uniqueness  of a strong solution.
The additional properties (I)-(III) follow from the proof of \cite[Theorem~6.1]{RoidosShao18}.
\end{proof} 

Before concluding this section, we will prove two useful properties of the solutions.
\begin{lem}\label{S4: Lyapunov stable semigroup}
Suppose that $u,\hat{u}$ are strong solutions to \eqref{S1: FPME} with respect to the initial data $u_0,\hat{u}_0$ obtained by the above argument. Then for any $t\geq 0$
$$
\| u(t)- \hat{u}(t)\|_1 \leq \| u_0- \hat{u}_0\|_1.
$$
\end{lem}
\begin{proof}
Proposition~\ref{S4: semilinear thm} implies that
$
\| u_{n;\omega}(t)- \hat{u}_{n;\omega}(t)\|_1 \leq \| u_0- \hat{u}_0\|_1,
$
and by \eqref{S5: mild sol def and converg}. 
$$
\| u_{\omega}(t)- \hat{u}_{\omega}(t)\|_1 \leq \| u_0- \hat{u}_0\|_1.
$$
The assertion then  follows from the  convergence of $u_{\omega}, \hat{u}_{\omega}$ to $u,\hat{u}$ in $C([0,T],L_1(\M))$.
\end{proof}

\begin{lem}\label{S4: phi(u_omega) L infinity}
Given $u_0\in L_1(\M)\cap L_\infty(\M)$, for any $0<\tau$, the strong solution $u$ to \eqref{S1: FPME} satisfies
$$
\|(-\Delta )^{\sigma/2}\Phi(u)(t)\|_2\leq M=M(\|u(\tau)\|_\infty)  \quad \text{for a.a. } \tau<t.
$$
\end{lem} 
\begin{proof}
The assertion can be proved by following the argument leading to \cite[(5.24)]{RoidosShao18}. 
\end{proof}

%%%%%%%%%%%%%%%%%%%%%%%%%%%%%%%%%%%%%%%

%%%%%%%%%%%%%%%%%%%%%%%%%%%%%%%%%%%%%%%%%%%%%%%%%%%%%%%%%

\section{\bf Asymptotic Behavior: Complete and Non-compact Manifolds}\label{Section 6}

In this and the next section, we always assume that $m>1$. 
We first consider the case $u_0 \in L_1(\M)\cap L_\infty(\M)$.
The initial condition  will be weakened in Section~\ref{Section 6.2}.

\subsection{\bf Asymptotic behavior for $u_0\in L_1(\M) \cap L_\infty(\M)$} 
\label{Section 6.1}

Given $\omega>0$, for the strong solution $u_{\omega}$ to \eqref{S4-FPME-omega}, we put
$$
\phi_r(t):= \| u_\omega(t) \|_r^r.
$$
Our aim is to derive an ordinary differential inequality for $\ln \|u_\omega(s)\|_{r(s)} $. 
Note that \cite[Proposition~4]{BreStr73}, Fatou's Lemma  and \eqref{S5: mild sol def and converg} show  that for all $p\in [1,\infty]$
\begin{equation}
\label{contraction discrete general}
  \|u_{n,k;\omega}\|_p, \|u_\omega\|_p \leq \|u_0\|_p.
\end{equation} 
Moreover,  \eqref{S4: DPME} implies that $
( -\Delta)^\sigma  \Phi(u_{n,k;\omega}) \in L_1(\M) \cap L_\infty(\M)
$.

%The fact that  $u_0\geq 0$ and the proof for \cite[Theorem~6.1(I)]{RoidosShao18} implies that $u_{n,k;\omega} \geq 0$.
For $r\geq 2$, we multiply  \eqref{S4: DPME} by $|u_{n,k;\omega}|^{r-2} u_{n,k;\omega}$ and integrate over $\M$. Putting $d=r+m-1$, this yields
\begin{align}
\notag & \int_\M |u_{n,k;\omega}|^{r-2} u_{n,k;\omega} (u_{n,k-1;\omega} - u_{n,k;\omega})\, d\mu_g\\
\notag=& \Delta T_k \Big[ \omega\|   u_{n,k;\omega}  \|_d^d + \int_\M  ( -\Delta)^\sigma  \Phi(u_{n,k;\omega}) |u_{n,k;\omega}|^{r-2} u_{n,k;\omega} \, d\mu_g \Big] \\
\label{S5: lower bd phi disc}
\geq & \Delta T_k \Big[ \omega \|   u_{n,k;\omega}  \|_d^d +  \frac{4m(r-1)}{d^2} \int_\M | ( -\Delta)^{\sigma/2} |u_{n,k;\omega}|^{d/2} |^2\, d\mu_g \Big].
\end{align}
We have used Lemma~\ref{S2: S-V ineq} in \eqref{S5: lower bd phi disc}.

Note that we have to start from \eqref{S4: DPME}  instead of \eqref{S1: FPME}, as in general we do not know whether $\Phi(u) \in D((-\Delta_{d/m})^\sigma)$.

As before, one can derive from the H\"older and the Young's inequalities that
\begin{align*}
& \int_\M |u_{n,k;\omega}|^{r-2} u_{n,k;\omega}  (u_{n,k-1;\omega} - u_{n,k;\omega})\, d\mu_g   \\
\leq &  \frac{1}{r} \int_\M  \Big[|u_{n,k-1;\omega}|^r -|u_{n,k ;\omega}|^r \Big] \, d\mu_g <M
\end{align*}
for some $M>0$ independent of $n$ and $\omega$ by \eqref{contraction discrete general}.

For $t\in (0,T)$ and $h>0$ small so that $t+h<T$, without loss of generality, we may assume $t=t_i$, $t+h=t_j$ for some $i,j\in \{1,2,\cdots,n-1\}$. We sum over all $[t_{k-1},t_k)$ contained in $[t,t+h)$ and obtain
\begin{align}
\notag
& \frac{4m(r-1)}{d^2} \int_t^{t+h} \int_\M  | ( -\Delta)^{\sigma/2} |u_{n;\omega}(s)|^{d/2}  |^2\, d\mu_g\, ds  \\
\label{S5: control disc}
\leq  & \frac{1}{r} \int_\M \Big[ |u_{n;\omega}(t)|^r -|u_{n;\omega}(t+h)|^r \Big] \, d\mu_g.
\end{align}
Hence due to \eqref{contraction discrete general}
\begin{align}
\notag &\int_t^{t+h} \int_\M  | (-\Delta)^{\sigma/2} |u_\omega(s)|^{d/2} )|^2\, d\mu_g\, ds \\
\label{S5: weak converg}
\leq &
 \liminf\limits_{n\to \infty} \int_t^{t+h} \int_\M  | ( -\Delta)^{\sigma/2} |u_{n;\omega}(s)|^{d/2}  |^2\, d\mu_g\, ds<M.
\end{align}
By the interpolation theory, \eqref{S5: mild sol def and converg} and \eqref{contraction discrete general}, for any $1<q<\infty$
\begin{align*}
\| u_{n;\omega}(t) -u_\omega(t)\|_q\leq \| u_{n;\omega}(t) -u_\omega(t)\|_1^{1/q} \|u_{n;\omega}(t) -u_\omega(t)\|_\infty^{1-1/q}\to 0
\end{align*}
as $n\to \infty$. Thus we can control the right hand side of \eqref{S5: control disc} and obtain
\begin{align}
\notag &\frac{4m(r-1)}{d^2} \int_t^{t+h} \int_\M  | (-\Delta)^{\sigma/2} |u_\omega(s)|^{d/2}  |^2\, d\mu_g\, ds \\
\label{S5: control u}
\leq &   \frac{1}{r} \int_\M [|u_\omega(t)|^r -|u_\omega(t+h)|^r] \, d\mu_g 
\end{align}
for a.a. $h>0$ small. 
Based on \eqref{contraction discrete general} and the Dominated Convergence Theorem, dividing both sides by $h$ and letting $h\to 0$ yields
\begin{equation}
\label{S5: der phi_r}
\frac{d}{dt} \phi_r(t) \leq - \frac{4m r(r-1)}{(r+m-1)^2} \| (-\Delta)^{\sigma/2} |u_\omega(t)|^{\frac{r+m-1}{2}} \|_2^2,
\quad r\geq 2.
\end{equation}
%Note that, at this stage, we cannot let $\omega\to 0$, as we do not know whether $u_\omega(t)\to u(t)$ in $L_r(\M)$ in general.

Given any $p\geq 2$, we define a $C^1$ and non-decreasing function $r:[0,t)\to [p,\infty)$ such that $r(0)=p$ and $\lim\limits_{s\to t^-} r(s)= +\infty$. Put $d(s)= r(s)+m-1$. 

We set $\Phi(r,s):= \|u_\omega(s)\|_r^r$. Then \eqref{S5: der phi_r} yields
\begin{align}
 \frac{d}{ds} \Phi(r(s),s) 
\notag =& \frac{\partial}{\partial s}\Phi(r ,s)|_{r=r(s)} + \dot{r}(s) \frac{\partial}{\partial r}\Phi(r ,s)|_{r=r(s)}\\
\notag  \leq & - \frac{4mr(s)(r(s)-1)}{d^2(s)} \| (-\Delta)^{\sigma/2} |u_\omega(s)|^{d(s)/2}\|_2^2 \\
\label{S5: Phi der}
&+ \dot{r}(s)\int_\M \ln(|u_\omega(s)|) |u_\omega(s)|^{r(s)}\, d\mu_g.
\end{align}
Defining 
$$Y(s):=\ln\|u_\omega(s)\|_{r(s)}$$ 
and following \cite{BonGri05}, we introduce the Young functional  $J:[1,\infty)\times X$, where $X=\bigcap_{p=1}^\infty L_p(\M)$ is defined by
$$
J(r,u):=\int_\M \ln  \Big(\frac{|u|}{\|u\|_r}  \Big) \frac{|u|^r}{\|u\|^r_r}\, d\mu_g.
$$ 
One can compute by using \eqref{S5: Phi der} that
\begin{align}
\notag \frac{d}{ds} Y(s)  = & -\frac{\dot{r}(s)}{r^2(s)} \ln \|u_\omega(s)\|^{r(s)}_{r(s)} + \frac{1}{r(s)\|u_\omega(s)\|_{r(s)}^{r(s)}} \frac{d}{ds} \Phi(r(s),s) \\
\notag  \leq & -\frac{\dot{r}(s)}{r^2(s)} \ln \|u_\omega(s)\|^{r(s)}_{r(s)} -   \frac{4m(r(s)-1)}{d^2(s)} \frac{ \| (-\Delta)^{\sigma/2} | u_\omega(s)|^{d(s)/2}\|_2^2 }{\|u_\omega(s)\|_{r(s)}^{r(s)}}\\
 & + \frac{\dot{r}(s)}{r(s)\|u_\omega(s)\|_{r(s)}^{r(s)}}\int_\M \ln(|u_\omega(s)|) |u_\omega(s)|^{r(s)}\, d\mu_g \\
\label{S5: equation 1}
=& \frac{\dot{r}(s)}{r(s)}J(r(s),u_\omega(s)) -   \frac{4m(r(s)-1)}{d^2(s)} \frac{ \| (-\Delta)^{\sigma/2}| u_\omega(s)|^{d(s)/2}\|_2^2 }{\|u_\omega(s)\|_{r(s)}^{r(s)}}.
\end{align}
We have used \cite[Proposition~2.6(a)]{BonGri05} in the last step.

Note that it follows from \eqref{S5: weak converg} that for a.a $t\in (0,T)$, $|u_\omega(t)|^{d(t)/2}\in D((-\Delta)^{\sigma/2})$. So by Theorem~\ref{Thm: functional ineq A1}(vi), it holds that
\begin{align*}
 & \|( -\Delta)^{\sigma/2} |u_\omega(s)|^{d(s)/2}\|_2^2   \\
\geq & \frac{2\sigma}{\widehat{C} \varepsilon n}\int_\M |u_\omega(s)|^{d(s)} \ln \Big(\frac{ |u_\omega(s)|^{d(s)}}{\|u_\omega(s)\|_{d(s)}^{d(s)}} \Big) \,d\mu_g + \frac{1}{\widehat{C}\varepsilon}\|u_\omega(s)\|_{d(s)}^{d(s)} \ln\varepsilon\\
=&  \frac{1}{\widehat{C} \varepsilon  }\|u_\omega(s)\|_{d(s)}^{d(s)} \left[\frac{2\sigma}{n} J(1,|u_\omega(s)|^{d(s)}) + \ln\varepsilon \right].
\end{align*}
where $\widehat{C}$ is the constant in Theorem~\ref{Thm: functional ineq A1}(iv), and we have used the equality 
$$\||u_\omega(s)|^{ d(s)/2}\|_2^2=\|u_\omega(s)\|_{d(s)}^{d(s)}.$$
Plugging this inequality into \eqref{S5: equation 1}, one can infer that
\begin{align*}
\notag \frac{d}{ds} Y(s) \leq & \frac{\dot{r}(s)}{r(s)}J(r(s),u_\omega(s))\\
& -   \frac{4m(r(s)-1)}{d^2(s)\widehat{C} \varepsilon} \frac{ \|u_\omega(s)\|_{d(s)}^{d(s)} }{\|u_\omega(s)\|_{r(s)}^{r(s)}}\left[\frac{2\sigma}{n} J(1,|u_\omega(s)|^{d(s)}) + \ln\varepsilon \right].
\end{align*}
%Now we have arrived at the inequality in \cite[(4.1)]{BonGri05}. 
Taking
$$
\varepsilon= \frac{4m}{n \widehat{C}} \frac{ r(s) [2\sigma r(s) +n(m-1)](r(s)-1)}{\dot{r}(s) d^2(s) } \frac{ \|u_\omega(s)\|_{d(s)}^{d(s)} }{\|u_\omega(s)\|_{r(s)}^{r(s)}}
$$
and using \cite[Proposition~2.6(b)]{BonGri05}, we have
\begin{align*}
  \frac{d}{ds} Y(s) \leq & \frac{\dot{r}(s)}{r^2(s)} \left[ J(1,|u_\omega(s)|^{r(s)}) - \frac{2\sigma r(s)}{2\sigma r(s) 
  + n(m-1)}J(1,|u_\omega(s)|^{d(s)}) \right] \\
& - \frac{\dot{r}(s)}{r^2(s)}   \frac{n  r(s)}{2\sigma r(s) + n(m-1)} \ln \frac{ \|u_\omega(s)\|_{d(s)}^{d(s)} }{\|u_\omega(s)\|_{r(s)}^{r(s)}} \\
& - \frac{\dot{r}(s)}{r (s)}   \frac{ n  }{2\sigma r(s) + n(m-1)} \ln \Big[\frac{4m}{n \widehat{C}} \frac{ r(s) [2\sigma r(s) +n(m-1)](r(s)-1)}{\dot{r}(s) d^2(s) } \Big].
\end{align*}
It follows from \cite[(4.3)]{BonGri05} and \cite[Proposition~2.6(b)]{BonGri05} that
\begin{align*}
 \frac{d}{ds} Y(s) \leq & \frac{\dot{r}(s)}{r^2(s)} \left[ J(1,|u_\omega(s)|^{r(s)}) - \frac{2\sigma r(s)}{2\sigma r(s) 
 + n(m-1)}J(1,|u_\omega(s)|^{d(s)}) \right] \\
& - \frac{\dot{r}(s)}{r^2(s)}   \frac{n  r(s) (m-1)}{2\sigma r(s) + n(m-1)} \Big[ J(r(s),u_\omega(s)) +Y(s) \Big] \\
& - \frac{\dot{r}(s)}{r (s)}   \frac{ n  }{2\sigma r(s) + n(m-1)} \ln \Big[\frac{4m}{n \widehat{C}} \frac{ r(s) [2\sigma r(s) +n(m-1)](r(s)-1)}{\dot{r}(s) d^2(s) } \Big]\\
=& \frac{\dot{r}(s)}{r (s)}\frac{2\sigma }{2\sigma r(s) + n(m-1)} \Big[ J(1,|u_\omega(s)|^{r(s)}) -  J(1,|u_\omega(s)|^{d(s)}) \Big]\\
& - \frac{\dot{r}(s)}{r (s)}   \frac{n  (m-1)}{2\sigma r(s) + n(m-1)}Y(s)\\
& - \frac{\dot{r}(s)}{r (s)}   \frac{ n   }{2\sigma r(s) + n(m-1)} \ln \Big[ \frac{4m}{n \widehat{C}} \frac{ r(s) [2\sigma r(s) +n(m-1)](r(s)-1)}{\dot{r}(s) d^2(s) } \Big].
\end{align*}
Taking into consideration \cite[Proposition~2.6(d)]{BonGri05} and $m>1$, we have
$$
J(1,|u_\omega(s)|^{r(s)}) -  J(1,|u_\omega(s)|^{d(s)}) \leq 0.
$$
Since $r$ is non-decreasing, by putting
$$
p(s)=\frac{\dot{r}(s)}{r (s)}   \frac{n  (m-1)}{2\sigma r(s) + n(m-1)}
$$ 
and
$$
q(s)=\frac{\dot{r}(s)}{r (s)}   \frac{ n  }{2\sigma r(s) + n(m-1)} \ln \Big[\frac{4m}{n \widehat{C}} \frac{ r(s) [2\sigma r(s) +n(m-1)](r(s)-1)}{\dot{r}(s) d^2(s) } \Big],
$$
we arrive at
$$
\frac{d}{ds}Y(s) + p(s)Y(s) + q(s) \leq 0,\quad Y(0)=\ln\|u_0\|_p.
$$
Hence $Y(s)\leq Y_L(s)$, where
$$
Y_L(s)=e^{-\int_0^s p(a)\, da} \Big[Y(0) - \int_0^s q(a) e^{\int_0^a p(\tau) \, d\tau} \, da \Big]
$$
is the solution of 
$$
\frac{d}{ds}Y_L(s) + p(s)Y_L(s) + q(s) = 0,\quad Y_L(0)=\ln\|u_0\|_p.
$$
By taking $r(s)=pt/(t-s)$, one can compute
\begin{align*}
P(s)&= \int_0^s p(a)\, da = \int_0^s \frac{\dot{r}(a)}{r (a)}   \frac{n  (m-1)}{2\sigma r(a) + n(m-1)}\,da\\
&=\ln \Big[\frac{2\sigma r(s)}{2\sigma r(s)+n(m-1)} \frac{2\sigma p + n(m-1)}{2\sigma p} \Big],
\end{align*}
and  it holds
$$
\lim\limits_{s\to t^-} e^{-P(s)}= \frac{2\sigma p}{2\sigma p + n(m-1)};
$$
and since  $\dot{r}(s)=\frac{r^2(s)}{pt}$, we further have
$$
q(a)e^{P(a)}= \frac{2\sigma \dot{r}(a)n [2\sigma p + n(m-1)]}{2\sigma p   [2\sigma r(a)+n(m-1)]^2} \ln \Big[\frac{4m}{n\widehat{C}} \frac{[2\sigma r(a) +n(m-1)](r(a)-1)}{r(a)[ r(a)+m-1]^2 } pt \Big].
$$
This implies
\begin{align*}
\lim\limits_{s\to t^-}\int_0^s q(a)e^{P(a)}\, da = R + \frac{n}{2\sigma p }\ln t
\end{align*}
for some $R=R(p,\sigma,m,n,\widehat{C})$ but independent of $\omega$. To sum up,  we have
$$
Y_L(t)=\frac{2\sigma p}{2\sigma p + n(m-1)} \ln\|u_0\|_p - \frac{n}{ 2\sigma p + n(m-1) }\ln t +R,
$$
and thus 
$$
\ln \| u_\omega(t)\|_\infty=\lim\limits_{s\to t^-} \ln \| u_\omega(s)\|_{r(s)}\leq \lim\limits_{s\to t^-}Y(s)\leq \lim\limits_{s\to t^-}Y_L(s)=Y_L(t).
$$
This yields
\begin{equation}
\label{S5: bound for u_omega}
\| u_\omega(t)\|_\infty \leq \frac{e^R}{t^\alpha}\|u_0\|_p^\gamma, 
\end{equation}
where $\displaystyle \alpha=\frac{n}{ 2\sigma p + n(m-1)}$ and $\displaystyle \gamma=\frac{2\sigma p}{2\sigma p + n(m-1)}$. Because the constants in \eqref{S5: bound for u_omega} are independent of $\omega$ and, for all $t$, $u_\omega(t)$ converges to $u(t)$ pointwise a.e. on $\M$, we immediately conclude that
\begin{equation}
\label{S5: bound for u}
\| u(t)\|_\infty \leq \frac{e^R}{t^\alpha}\|u_0\|_p^\gamma,\quad p\geq 2,
\end{equation}
where $u$ is the unique strong solution to \eqref{S1: FPME}.

%%%%%%%%%%%%%%%%%%%%%%%%%%%%%%%%%%%%%%%%%%%%%%%%%%%%%%%%%%%%%%

\subsection{\bf Proof of Theorem~\ref{S5: Large time behavior}}\label{Section 6.2}

\begin{proof}
(of Theorem~\ref{S5: Large time behavior})
Given  $  u_0 \in L_1(\M)\cap L_2(\M) $,  
we take a sequence $L_1(\M)\cap L_\infty(\M)\ni u_{0,k} \to u_0$ in $L_1(\M)\cap L_2(\M)$ and denote the corresponding strong solutions to \eqref{S1: FPME} by $u_k$. We learn from Lemma~\ref{S4: Lyapunov stable semigroup} that $(u_k)_k$ is Cauchy in $C([0,T],L_1(\M))$ and thus converges to some $u\in C([0,T],L_1(\M))$ for any $T>0$.

For every $0<\tau<\infty$, it follows from \eqref{S5: bound for u}  that $u_k \in L_\infty([\tau ,\infty),L_\infty(\M))$ with uniform bounds.
By the interpolation theory, $\|u_k(\tau)\|_{m+1}$ is uniformly bounded in $k$.
\eqref{S4: unif est 3} implies that 
$$
(-\Delta)^{\sigma/2}\Phi(u_k)\in L_2([\tau,\infty),L_2(\M))
$$
with uniform bound. 
Now we can pass the limit $k\to \infty$ in 
\begin{align*}
 \int^\infty_\tau\int_\M u_k  \partial_t \phi \, d\mu_g dt + \int_\M u_k(\tau)\phi(\tau)\, d\mu_g   
=  \int_\tau^\infty \int_\M ( -\Delta )^{\sigma/2} \Phi(u_k)  ( -\Delta )^{\sigma/2}  \phi \, d\mu_g dt   
\end{align*}
for any $\phi\in C^1_c([0,\infty)\times\M)$,
and infer that $u$ is a weak solution to \eqref{S1: FPME} on $[\tau,\infty)$. 
Since $u(0)=u_0$ and $u\in C([0,\infty),L_1(\M))$, 
\begin{align*}
  &\int_0^\infty \int_\M   ( -\Delta )^{\sigma/2} \Phi(u )  ( -\Delta )^{\sigma/2}  \phi \, d\mu_g dt   \\
=& \lim\limits_{\tau\to 0^+} \int_\tau^\infty \int_\M   ( -\Delta )^{\sigma/2} \Phi(u )  ( -\Delta )^{\sigma/2} \, d\mu_g dt \\
=& \lim\limits_{\tau\to 0^+} \int^\infty_\tau\int_\M u   \partial_t \phi \, d\mu_g dt + \lim\limits_{\tau\to 0^+}\int_\M u (\tau)\phi(\tau)\, d\mu_g   \\
= & \int^\infty_0\int_\M u   \partial_t \phi \, d\mu_g dt  +\int_\M u_0\phi(0)\, d\mu_g  .
\end{align*}
Thus $u$ is a weak solution to \eqref{S1: FPME} on $[0,\infty)$.
To see $u$ is indeed a strong solution, it suffices to observe that $u  \in L_\infty([\tau,\infty)\times 
\M)$ for any $\tau>0$. Then the proof of Proposition~\ref{Prop: exist strong sol} is still valid.
The uniqueness of   solution follows from   Lemma~\ref{lem: unique strong sol}.
By the approximation argument above, \eqref{S5: bound for u} still holds true for $u$.
\end{proof}

%%%%%%%%%%%%%%%%%%%%%%%%%%%%%%%%%%%%%%%%%%%%%%%%%%%%%%%%%

\section{\bf Asymptotic Behavior: Closed Manifolds}\label{Section 7}

\subsection{\bf Large time behavior for $u_0\in   L_\infty(\M)$} 
\label{Section 7.1}

The argument for the closed manifold of dimension $n>2$ is very similar to that in Section~\ref{Section 6.1} and thus we will only point out necessary modifications. First note that \eqref{S5: equation 1} still holds true. Then applying Theorem~\ref{Thm: functional ineq A2}(vi) yields
\begin{align*}
\| (-\Delta)^{\sigma/2} v\|^2_2 & \geq \frac{2\sigma}{n M_0 \varepsilon} \int_\M |v|^2 \ln \Big(\frac{|v|}{\|v\|_2^2} \Big)^2 + \frac{1}{M_0 \varepsilon} \|v\|_2^2 \ln \varepsilon - \frac{M_1}{M_0} |\overline{v}|^2.
\end{align*}
Using this inequality in \eqref{S5: equation 1} and setting $m_0= \max\{ m-1, 1\}$, we conclude that
\begin{align}
\notag \frac{d}{ds} Y(s) \leq & \frac{\dot{r}(s)}{r(s)} J(r(s), u_\omega(s)) - \frac{4m(r(s)-1)}{d^2(s)M_0 \varepsilon} \frac{2\sigma}{n } \frac{\|u_\omega(s)\|_{d(s)}^{d(s)}}{ \|u_\omega(s)\|_{r(s)}^{r(s)} } J(1, |u_\omega(s)|^{d(s)} ) \\
\notag  &- \frac{4m(r(s)-1)}{d^2(s) M_0 \varepsilon}\frac{\|u_\omega(s)\|_{d(s)}^{d(s)}}{ \|u_\omega(s)\|_{r(s)}^{r(s)} } \ln \varepsilon + \frac{4m(r(s)-1) M_1}{d^2(s) M_0} \frac{|\overline{|u_\omega(s)|^{d(s)/2 }}|^2}{ \|u_\omega(s)\|_{r(s)}^{r(s)} }\\
\notag  =& \frac{\dot{r}(s)}{r(s)} J(r(s), u_\omega(s)) -  \frac{4m(r(s)-1)}{d^2(s)M_0 \varepsilon}  \frac{\|u_\omega(s)\|_{d(s)}^{d(s)}}{ \|u_\omega(s)\|_{r(s)}^{r(s)} } [\frac{2\sigma}{n } J(1, |u_\omega(s)|^{d(s)} ) + \ln \varepsilon]\\
\notag  & + \frac{4m(r(s)-1) M_1}{d^2(s) M_0} \frac{|\overline{|u_\omega|^{d(s)/2 }(s)}|^2}{ \|u_\omega(s)\|_{r(s)}^{r(s)} }\\
 \notag   \leq & \frac{\dot{r}(s)}{r(s)} J(r(s), u_\omega(s)) -  \frac{8 \sigma m (r(s)-1)}{M_0 \varepsilon d(s) n} \frac{\|u_\omega(s)\|_{d(s)}^{d(s)}}{ \|u_\omega(s)\|_{r(s)}^{r(s)} }[J(d(s), u_\omega(s)) + \frac{n \ln \varepsilon}{2\sigma d(s)} ] \\
 \label{S6: eq 1}
&  +  \frac{4m(r(s)-1) M_1}{d^2(s) M_0} \|u_0\|_{m_0}^{m-1}.
\end{align}
In \eqref{S6: eq 1}, we have used the equality
$
\big|\overline{|v|^{d/2 } } \big|^2 = \|v\|_{d/2}^d,
$ 
the interpolation inequality
$$
\|v\|_{d/2}^d  \leq \|v\|_{m-1}^{m-1} \|v\|_r^r ,\quad r \geq m-1,
$$
the contraction property 
$$
\|u_\omega (s)\|_q \leq \|u_0\|_q ,\quad s\geq 0, \, q\in [1,\infty],
$$
which follows from the proof of \cite[Theorem~6.1(II)]{RoidosShao18}, and the reverse H\"older inequality (when $m<2$).

Taking 
$$
\varepsilon= \frac{8\sigma m r(s)(r(s)-1)}{\dot{r}(s) M_0 d(s) n} \frac{\|u_\omega(s)\|_{d(s)}^{d(s)}}{ \|u_\omega(s)\|_{r(s)}^{r(s)} },
$$
we infer that 
\begin{align*}
\frac{d}{ds} Y(s) \leq & \frac{\dot{r}(s)}{r(s)} \left[ J(r(s), u_\omega(s)) - J(d(s), u_\omega(s)) -\frac{n}{2\sigma d(s)} \ln (\frac{\|u_\omega(s)\|_{d(s)}^{d(s)}}{ \|u_\omega(s)\|_{r(s)}^{r(s)} }) \right]\\
&- \frac{\dot{r}(s)}{r(s)} \frac{n}{2\sigma d(s)}  \ln (\frac{8\sigma m r(s) (r(s)-1)}{\dot{r}(s) M_0 d(s) n })+  \frac{4m(r(s)-1) M_1}{d^2(s) M_0} \|u_0\|_{m_0}^{m-1}.
\end{align*}
It follows from \cite[Proposition~2.6(a)]{BonGri05} that
$$
J(r(s), u_\omega(s)) - J(d(s), u_\omega(s)) - \ln \Big(\frac{\|u_\omega(s)\|_{d(s)} }{ \|u_\omega(s)\|_{r(s)}  } \Big) \leq 0.
$$
This yields
\begin{align}
\notag \frac{d}{ds} Y(s) \leq &  \frac{\dot{r}(s)}{r(s)} \left[\ln (\frac{\|u_\omega(s)\|_{d(s)}}{ \|u_\omega(s)\|_{r(s)}  }) -\frac{n}{2\sigma d(s)} \ln (\frac{\|u_\omega(s)\|_{d(s)}^{d(s)}}{ \|u_\omega(s)\|_{r(s)}^{r(s)} }) \right] \\
\notag & - \frac{\dot{r}(s)}{r(s)} \frac{n}{2\sigma d(s)}  \ln \Big(\frac{8\sigma m r(s) (r(s)-1)}{\dot{r}(s) M_0 d(s) n } \Big)+  \frac{4m(r(s)-1) M_1}{d^2(s) M_0} \|u_0\|_{m_0}^{m-1} \\
\notag = & \frac{\dot{r}(s)}{r(s)} \left[ (1-\frac{n}{2\sigma}) \ln \|u_\omega(s)\|_{d(s)} + ( \frac{nr(s)}{2 \sigma d(s)} -1) \ln \|u_\omega(s)\|_{r(s)}  \right] \\
\notag & - \frac{\dot{r}(s)}{r(s)} \frac{n}{2\sigma d(s)}  \ln \Big(\frac{8\sigma m r(s) (r(s)-1)}{\dot{r}(s) M_0 d(s) n } \Big)+  \frac{4m(r(s)-1) M_1}{d^2(s) M_0} \|u_0\|_{m_0}^{m-1} \\
\label{S6: eq 2}
\leq & \frac{\dot{r}(s)}{r(s)} \left[ (1-\frac{n}{2\sigma}) \ln \|u_\omega(s)\|_{r(s)} + ( \frac{nr(s)}{2 \sigma d(s)} -1) \ln \|u_\omega(s)\|_{r(s)}  \right] \\
\notag & - \frac{\dot{r}(s)}{r(s)} \frac{n}{2\sigma d(s)}  \ln \Big(\frac{8\sigma m r(s) (r(s)-1)}{\dot{r}(s) M_0 d(s) n } \Big)+  \frac{4m(r(s)-1) M_1}{d^2(s) M_0} \|u_0\|_{m_0}^{m-1} \\
\notag = & -\frac{\dot{r}(s)}{r(s)}\frac{n}{2\sigma  } \Big(1- \frac{r(s)}{d(s)} \Big) \ln \|u_\omega(s)\|_{r(s)}\\
\notag &  - \frac{\dot{r}(s)}{r(s)} \frac{n}{2\sigma d(s)}  \ln \Big(\frac{8\sigma m r(s) (r(s)-1)}{\dot{r}(s) M_0 d(s) n } \Big)  
\notag   +  \frac{4m(r(s)-1) M_1}{d^2(s) M_0} \|u_0\|_{m_0}^{m-1} .
\end{align}
Here \eqref{S6: eq 2} follows from the facts $n>2\sigma$, ${\rm vol}(\M)=1$ and H\"older inequality
$$
\|u_\omega \|_r \leq \|u_\omega \|_d.
$$
We put 
$$
p(s)= \frac{\dot{r}(s)}{r(s)}\frac{n (m-1)}{2\sigma d(s) }  
$$
and
$$
q(s)=\frac{\dot{r}(s)}{r(s)} \frac{n}{2\sigma d(s)}  \ln \Big(\frac{8\sigma m r(s) (r(s)-1)}{\dot{r}(s) M_0 d(s) n } \Big)-  \frac{4m(r(s)-1) M_1}{d^2(s) M_0} \|u_0\|_{m_0}^{m-1}.
$$
Then we obtain the following differential inequality 
$$
\frac{d}{ds}Y(s) + p(s)Y(s) + q(s) \leq 0,\quad Y(0)=\ln\|u_0\|_p.
$$
As in Section~\ref{Section 6.1}, we have
$$
Y(s)\leq Y_L(s)=e^{-\int_0^s p(a)\, da} \Big[Y(0) - \int_0^s q(a) e^{\int_0^a p(\tau) \, d\tau} \, da \Big].
$$
One computes
\begin{align*}
P(s)= & \int_0^s p(a)\, da = \int_0^s  \frac{\dot{r}(a)}{r(a)}\frac{n (m-1)}{2\sigma (r(a) +m-1) }  \, da \\
= & \frac{n  }{2\sigma} \ln \Big[ \frac{r(s)}{d(s)} \frac{p+m-1}{p} \Big].
\end{align*}
This implies
$$
e^{-P(t)}= \lim\limits_{s\to t^-} e^{-P(s)} = \Big( \frac{p}{p+m-1} \Big)^{n/2\sigma}.
$$
Following \cite{BonGri05}, we have
\begin{align*}
Q(s)=& \int_0^s q(a) e^{\int_0^a p(\tau) \, d\tau} \, da \\
= & \underbrace{\int_0^s \frac{\dot{r}(a)}{r(a)} \frac{n}{2\sigma d(a)}  \ln  \Big(\frac{8\sigma m  }{  M_0  n } \Big)  \Big( \frac{r(a)}{d(a)} \Big)^{n/2\sigma}   \Big( \frac{p+m-1}{p} \Big)^{n/2\sigma}\, da}_{Q_1(s)}\\
& + \underbrace{\int_0^s \frac{\dot{r}(a)}{r(a)} \frac{n}{2\sigma d(a)}  \ln  \Big(\frac{r(s) (r(s)-1) }{ \dot{r}(s) d(s) } \Big)  \Big( \frac{r(a)}{d(a)} \Big)^{n/2\sigma}   \Big( \frac{p+m-1}{p}  \Big)^{n/2\sigma}\, da}_{Q_2(s)}\\
& - \underbrace{\int_0^s \frac{4m(r(s)-1) M_1}{d^2(s) M_0} \|u_0\|_{m_0}^{m-1}  \Big( \frac{r(a)}{d(a)} \Big)^{n/2\sigma}   \Big( \frac{p+m-1}{p}  \Big)^{n/2\sigma}\, da}_{Q_3(s)}.
\end{align*}
Letting
$$
I_1=I_1(m,p,n,\sigma)=\lim\limits_{s\to t^-} \int_0^s \frac{\dot{r}(a)}{r(a)d(a)}   \Big( \frac{r(a)}{d(a)} \Big)^{n/2\sigma}   \, da
$$
and
$$
I_2=I_2(m,p,n,\sigma)=\lim\limits_{s\to t^-}  \int_0^s \frac{\dot{r}(a)}{r(a)d(a)}   \Big( \frac{r(a)}{d(a)} \Big)^{n/2\sigma} \ln \Big(\frac{r(a)-1}{r(a) d(a)} \Big)   \, da
$$
gives
$$
Q_1(t)= I_1 \frac{n}{2\sigma  }  \Big(\frac{p+m-1}{p} \Big)^{n/2\sigma} \ln  \Big(\frac{8\sigma m  }{  M_0  n } \Big)     
$$
and
$$
Q_2(t)= I_1 \frac{n}{2\sigma  }  \Big(\frac{p+m-1}{p} \Big)^{n/2\sigma}  \ln(pt) +I_2   \frac{n}{2\sigma  }  \Big(\frac{p+m-1}{p} \Big)^{n/2\sigma} 
$$
as $\dot{r}(s)=\frac{r^2(s)}{pt}$.
Note that $1<  r(s)\leq d(s) $ for all $0\leq s < t$. We thus conclude that
\begin{align*}
Q_3(t)= & \frac{4m M_1}{M_0}  \Big(\frac{p+m-1}{p} \Big)^{n/2\sigma}    \|u_0\|_{m_0}^{m-1} \lim\limits_{s\to t^-}  \int_0^s \frac{r(a)-1}{d^2(a)} \Big( \frac{r(a)}{d(a)} \Big)^{n/2\sigma} \, da  \\
\leq & \frac{4m M_1}{M_0}  \Big(\frac{p+m-1}{p} \Big)^{n/2\sigma}    \|u_0\|_{m_0}^{m-1}  t.
\end{align*} 
To sum up, by putting
$$
\gamma= \Big(\frac{p}{p+m-1} \Big)^{n/2\sigma},  \quad   \alpha=I_1 \frac{n}{2\sigma  }    \quad \text{and} \quad  E=\frac{4m M_1}{M_0}   ,
$$
we finally arrive at
\begin{align*}
Y_L(t)\leq   \gamma \ln\|u_0\|_p + R(m,p,n,\sigma,M_0)-\alpha \ln(pt) + E\|u_0\|_{m_0}^{m-1}  t 
\end{align*}
for some constant $R(m,p,n,\sigma,M_0)$.  
This yields
\begin{equation}
\label{S6: asym for u omega}
\| u_\omega(t)\|_\infty \leq C \frac{e^{E \|u_0\|_{m_0}^{m-1}  t }}{t^\alpha}\|u_0\|_p^\gamma,\quad p\geq 2, 
\end{equation} 
and further
\begin{equation}
\label{S6: asym for u 1}
\| u(t)\|_\infty \leq C \frac{e^{E \|u_0\|_{m_0}^{m-1}  t }}{t^\alpha}\|u_0\|_p^\gamma,\quad p\geq 2. 
\end{equation}

%%%%%%%%%%%%%%%%%%%%%%%%%%%%%%%%%%%%%%%%%%%%%%%%%%%%%%%%%%%%%%%%%%%%%%%%

\subsection{\bf Passing to initial data in $ L_2(\M)$}
\label{Section 7.2}

We can follow the approximation procedure in Section~\ref{Section 6.2} and use a sequence $(u_{0,k})_k \in L_\infty(\M)$ to approximate an initial datum $u_0\in L_2(\M)$  in $L_2(\M)$. The corresponding  solutions   $(u_k)_k$ then converge to a strong solution $u$ to \eqref{S1: FPME}. This solution is unique due to Lemma~\ref{lem: unique strong sol}. \eqref{S6: asym for u 1} still holds true for $u$.

%%%%%%%%%%%%%%%%%%%%%%%%%%%%%%%%%%%%%%%%%%%%%%%%%%%%%%%%%%%%%%%%%%%%%%%%

\subsection{\bf Convergence of  solution for general initial data} 
\label{Section 7.3}

In this section, we assume that $u_0 \in L_2(\M)$ whose mean is not necessarily zero. 
We  prove that the unique strong solution $u$ to \eqref{S1: FPME} converges to the average of $u_0$ in $L_p(\M)$ for any $1\leq p <\infty$. The idea is based on the theory in \cite{Dafer78}.

%\begin{lem}\label{S6: lem positive initial data Lyapunov}
%Assume that $0<\alpha < u_0 < M<\infty$. Then the solution $u$ to \eqref{S1: FPME} satisfies
%$$
%\lim\limits_{t\to \infty} \| u(t)-  \frac{1}{{\rm vol}(\M)} \int_\M u_0 \, d\mu_g \|_\infty =0.
%$$
%\end{lem}
\begin{theorem}
\label{S6: convergence non-zero mean}
Assume that $m>1$.
Let $u_0\in L_2(\M)$ and $u$ be the unique strong solution to \eqref{S1: FPME}. Then for all $p\in [1,\infty)$
\begin{equation}
\label{S6: eq Lp convergence}
\lim\limits_{t\to \infty} \| u(t)-  \frac{1}{{\rm vol}(\M)} \int_\M u_0 \, d\mu_g \|_p =0.
\end{equation}
\end{theorem}
\begin{proof}
Define the metric space $X=L_2(\M)$ equipped with the $L_1-$norm.
Note that \eqref{S1: FPME} is associated with a continuous (nonlinear) semigroup $T(t)$ in $X$, defined by $T(t)u_0=u(t;u_0)$, where $u(t;u_0)$ is the unique strong solution to \eqref{S1: FPME} with initial datum $u_0$. 
It follows from Lemma~\ref{lem: unique strong sol} that $T(\cdot)$ is Lyapunov stable in the sense of \cite[Definition~4.1]{Dafer78}.

Denote the trajectory of $u_0$ by $\gamma(u_0)=\bigcup\limits_{t\geq 0} T(t)u_0$. The closure $\overline{\gamma(u_0)}$ in $L_1(\M)$ satisfies
$$
\overline{\gamma(u_0)}= \gamma(u_0) \cup \omega(u_0),
$$
as $u\in C([0,\infty), L_1(\M))$.
Here $\omega(u_0)$ is the $L_1(\M)$ $\omega-$limit set of $u_0$.
Let us characterize $\omega(u_0)$.
If there exists a sequence $(t_n)_n$ such that  $t_n\to \infty$ and
$$
u(t_n) \to w \in \omega(u_0) 
$$
in $L_1(\M)$,
then $(u(t_n))_n$ is Cauchy in $L_1(\M)$. 
In view of Theorem~\ref{S4: global weak sol}(II) and \eqref{S6: asym for u 1}, $(u(t_n))_n$ is uniformly bounded in $L_\infty(\M)$.  
By the Riesz-Thorin interpolation theorem, we immediately have that $(u(t_n))_n$ is Cauchy in $L_2(\M)$ and thus $w\in L_2(\M)$.
It follows from the discussion in the previous subsection that $T(t)w$ is a strong solution to \eqref{S1: FPME} with initial data $w$, and further its trajectory is Lyapunov stable in $L_1(\M)$.

In addition, it is clear that $\omega(u_0)=\omega(u(\tau))$ for any $\tau>0$. 
From Lemma~\ref{S4: phi(u_omega) L infinity}, Theorem~\ref{S4: global weak sol}(II) and \eqref{S6: asym for u 1}, we learn that $\Phi[(\gamma(u(\tau))]$ is bounded in $D((-\Delta)^{\sigma/2})$. 
Recall
$$
D((-\Delta)^{\sigma/2})= [L_2(\M), H^2_2(\M)]_{\sigma/2} \doteq H^\sigma_2(\M).
$$
By the Sobolev embedding, $\Phi[(\gamma(u(\tau))]$ is relatively compact in $L_1(\M)$.
Note that, by Theorem~\ref{S4: global weak sol}(II), $\overline{\Phi[(\gamma(u(\tau))]}$ is bounded in $L_\infty(\M)$.
We will show that 
\begin{equation}\label{S6: equi trajectory}
\overline{\Phi[(\gamma(u(\tau))]}= \Phi(\overline{\gamma(u(\tau))}).
\end{equation}
Indeed, since $\beta\in BC^{1/m}(\R)$ for $m>1$, we have
\begin{equation}\label{S6: beta}
\|u-v\|_1 \leq C\int_\M |\Phi(u)-\Phi(v)|^{1/m}\, d\mu_g \leq C \|\Phi(u)-\Phi(v)\|_1^{1/m};
\end{equation}
and due to $\Phi\in C^{1-}(\R)$,
\begin{equation}\label{S6: Phi}
 \|\Phi(u)-\Phi(v)\|_1 \leq C \|u-v\|_1 ,\quad u,v\in L_\infty(\M).
\end{equation}
It is clear that \eqref{S6: beta} and \eqref{S6: Phi} imply \eqref{S6: equi trajectory}. We can further derive from \eqref{S6: equi trajectory} and \eqref{S6: beta}   that 
$\overline{\gamma(u(\tau))}$ is sequentially compact in $L_1(\M)$ and thus $\gamma(u(\tau))$ is relatively compact in $L_1(\M)$.

We define
$$
V(\xi)=\frac{1}{m+1} \int_M |\xi|^{m+1}\, d\mu_g, \quad  \xi \in L_1(\M).
$$
Then $V$ is lower semicontinuous on $L_1(\M)$. 
It follows from Theorem~\ref{S4: global weak sol}(II) that $V$ is non-increasing along the orbit $\gamma(u_0)$  and thus is a Lyapunov functional for $T(\cdot) $ defined in  \cite{Dafer78}.
We  infer from \cite[Proposition~4.1]{Dafer78} that $V$ is constant on $\omega(u_0)$.

Take any $w_0\in \omega(u_0)$. 
Note that for any $t>0$, $w(t):=T(t)w_0 \in \omega(u_0)$ by the Lyapunov stability of $T(\cdot)$.
Recall that $w_0\in L_2(\M)$. 
%Note also that $w_0>0$ a.e. by Theorem~\ref{S4: global weak sol}(I).  
For any $T>\tau>0$, 
$$
\partial_t w , (-\Delta)^\sigma \Phi(w) \in L_\infty((\tau,T), L_1(\M))),
$$
and thus 
$$
\frac{d}{dt} V(w(t))= - \int_\M | (-\Delta)^{\sigma/2} \Phi(w)(t)|^2\, d\mu_g =0, \quad t>0.
$$
Thus $(-\Delta)^{\sigma/2} \Phi(w)(t)=0$ a.e.. This shows that
$$
-\Delta \Phi(w)(t) =(-\Delta)^{1-\sigma/2} (-\Delta)^{\sigma/2} \Phi(w)(t)=0,
$$
which implies that
$w(t)\equiv W$ for some constant $W$ and all $t\geq 0$ in view of the fact that $w\in C([0,\infty),L_1(\M))$. 
So the $L_1(\M)$ $\omega-$limit set of $u_0$ consists of constant functions. 
Because of the mass conservation property, cf. Theorem~\ref{S4: global weak sol}(III), we obtain 
$$
W=  \frac{1}{{\rm vol}(\M)} \int_\M u_0\, d\mu_g.
$$
This establishes \eqref{S6: eq Lp convergence} for $p=1$. The general case $p>1$ follows from the case $p=1$ and the $L_\infty$--contraction property,   cf. Theorem~\ref{S4: global weak sol}(II).
\end{proof}

%%%%%%%%%%%%%%%%%%%%%%%%%%%%%%%%%%%%%%%%%%%%%%%%%%%%%%%%%%%%%%

\subsection{\bf Asymptotic   behavior for $u_0$ with zero mean} 
\label{Section 7.4}

Now we consider a special case   $u_0\in L_2(\M)$ with zero mean, i.e. $\int_\M u_0\, d\mu_g=0$ and prove a refined asymptotic estimate.  
In the sequel, let $p \in  [2,\infty)$ and $t>2$.
We will closely follow the proof of \cite[Corollary~1.3]{BonGri05}.
First, \eqref{S5: der phi_r} gives
$$
\frac{d}{dt} \|u_\omega(t)\|_p^p \leq - \frac{4mp(p-1)}{d^2} \| (-\Delta)^{\sigma/2} |u_\omega(t)|^{d/2}\|_2^2,
$$
where $d=p+m-1$ as before.
Note that $u_\omega$ has zero mean for all $t\geq 0$. 
Following the proof of \cite[Lemma~3.2]{AlikRost81}, we can show that
$$
K \|   |u_\omega(t)|^{d/2}\|_2^2 \leq \| (-\Delta)^{\sigma/2} |u_\omega(t)|^{d/2}\|_2^2
$$
for some $K=K(m,p)$. This yields
$$
\frac{d}{dt} \|u_\omega(t)\|_p^p  \leq -B \|u_\omega(t)\|_d^d \leq -B \|u_\omega(t)\|_p^d 
$$
for some constant $B=B(p,m)$.
Setting  $\phi(t)= \|u_\omega(t)\|_p^p$, we thus obtain
$$
\frac{d}{dt} \phi(t)  \leq -B \phi(t)^{d/p}
$$
and thus
$$
\|u_\omega(t)\|_p \leq  \Big(\frac{1}{B t + \|u_0\|_p^{-(m-1)}} \Big)^{1/(m-1)},
$$
which also holds for $u$. When $t>2$, by \eqref{S6: asym for u 1} and Theorem~\ref{S4: global weak sol}(II)
$$
\|u(t)\|_\infty \leq  C  e^{E \|u(t-1)\|_{m_0}^{m-1}    } \|u(t-1)\|_p^\gamma \leq  \Big[\frac{C}{B (t-1) + \|u_0\|_p^{-(m-1)}} \Big]^{\gamma/(m-1)} 
$$
for some $C=C(p,m,n,\sigma,M_1,M_0)$. Moreover, for   any $\varepsilon\in (0,1)$, it follows from  the inequality $a^\varepsilon b^{1-\varepsilon}\leq a+b$ that
$$
\|u(t)\|_\infty \leq \frac{C \|u_0\|_p^{\varepsilon \gamma} }{[B (t-1)]^{\gamma(1-\varepsilon)/(m-1)}}.
$$
This completes the proof for Theorem~\ref{S6: Strong solution L 2 initial data}.

%%%%%%%%%%%%%%%%%%%%%%%%%%%%%%%%%%%%%%%%%%%

\appendix

\section{$m-$accretivity of the operator $[\omega+ (-\Delta_1)^\sigma]\Phi(u)$}\label{Section A}

Let $X_\R$ be a real Banach lattice with an order $\leq$. See \cite[Chapter~C-I]{ArenGrohNage86}. The complexification of $X_\R$ is a complex Banach lattice defined by
\begin{equation}
\label{S4.1: Banach lattice}
X:=X_\R \oplus i X_\R. 
\end{equation}
The positive cone of $X_\R$ is defined by
$$X_\R^+:=\{x\in X_\R:\, 0\leq x\}. $$
\begin{definition}
Let $\vartheta\in\R$, and $X$ be a complex Banach lattice defined as in \eqref{S4.1: Banach lattice}.
A semigroup $\{T(t)\}_{t\geq 0 } $ is called real if 
$$T(t) X_\R \subset X_\R ,\quad t\geq 0.$$ 
Further, we say that $\{T(t)\}_{t\geq 0 } $ is positive if 
$$T(t) X_\R^+ \subset X_\R^+  ,\quad t\geq 0.$$ 
\end{definition}

\begin{definition}\label{Def: Mark semigroup}
A strongly continuous semigroup $\{T(t)\}_{t\geq 0 } $ on $L_2(\M)$ is called  a Markov semigroup if it is both positive and $L_\infty$-contraction, i.e.  
$$
\|T(t) u\|_\infty \leq \|u\|_\infty,\quad t\geq 0, \quad u\in  L_\infty(\M) \cap L_2(\M).
$$
\end{definition}

Recall that $\Phi(x)=|x|^{m-1}x$ and   $\beta=\Phi^{-1}$. 
%Applying \cite[Theorem~1]{BreStr73}, we have the following proposition.
\begin{prop}
\label{S4: semilinear thm}
Let $\omega,\lambda>0$ and $\sigma\in (0,1)$.
For any $f\in L_1(\M)$, there exists a unique solution $u\in D((-\Delta_1)^\sigma)$ to 
$$ \lambda[\omega+(-\Delta_1 )^\sigma ] u +\beta(u) =f .$$
Moreover, for any $f_1,f_2\in L_1(\M)$, the corresponding solutions $u_1,u_2$ satisfy  
$$
\|\beta(u_1)  - \beta(u_2)  \|_1\leq \|f_1-f_2\|_1.
$$
\end{prop}
\begin{proof}
Following a similar argument to the proof of \cite[Lemma~4.1]{RoidosShao18}, one can show that given any $v\in L_1(\M)$,
$$ 
\sup [\id +\lambda (\omega+ (- \Delta_1)^\sigma)]^{-1}v \leq \max \{0, \sup v\}.
$$
Proposition~\ref{S2: Lap-frac invert} implies that there exists some $C>0$ such that for all $v\in L_1(\M)$
$$
C\|v\|_1 \leq \|[\omega+(-\Delta_1)^\sigma] v\|_1, \quad v\in D(( -\Delta_1)^\sigma).
$$
Now the proposition is a direct consequence of \cite[Theorem~1]{BreStr73}.
\end{proof}

%%%%%%%%%%%%%%%%%%%%%%%%%%%%%%%%%%%%%%%%%%%%%%

\begin{definition}\cite[Chapter II.3]{Bar73}
$\cA: D(\cA)\subset X \to X$ is a  nonlinear operator defined in a Banach space $X$.
\begin{itemize}
\item[(i)]  $\cA$ is called accretive if for all $\lambda>0$
$$
\|(\id +\lambda \cA) x_1 - (\id +\lambda \cA) x_2\|_X\geq \|x_1-x_2\|_X,\quad x_1,x_2\in D(\cA).
$$
\item[(ii)]  $\cA$ is called $m$-accretive if $\cA$ is accretive and it satisfies the range condition
$$
Rng(\id +\lambda \cA)=X, \quad \lambda>0.
$$
\end{itemize}
\end{definition}

\begin{prop}
\label{Prop: m-accretive}
The operator $[u\mapsto \cA(u):=[\omega+(-\Delta_1)^\sigma]\Phi(u) ]:D(\cA)\subset L_1(\M)\to L_1(\M)$ is $m$-accretive   with domain 
$$
D(\cA)=\{u\in L_1(\M): \Phi(u)\in D((-\Delta_1)^\sigma)\}  \quad \text{dense in  } L_1(\M).
$$ 
\end{prop}
%Then by the strict monotonicity of $\Phi$, it is a bijection from $D(\cA)$ to $D((-\Delta_{g,1})^\sigma)$. 
%Note that $u\in D(\cA)$ iff $\Phi(u)\in D((-\Delta_1)^\sigma)\subset \cH^{0,\gamma_1}_1(\M)$. This shows that $u\in \cH^{0,\gamma_m}_m(\M)\subset \cH^{0,\gamma_1}_1(\M)$.
\begin{proof}
Proposition~\ref{S4: semilinear thm} implies that  $\cA $ is $m$-accretive.  Next, we will prove that $D(\cA)$ is dense in $L_1(\M)$, which is clearly true for $m=1$.

Case 1: $m> 1$.  Observe that in this case $\beta\in BC^{1/m}(\R)$.
%Because $C_0^\infty(\M)\subset D( \Delta_{g,1})\subset D((- \Delta_{g,1})^\sigma)$, $D((- \Delta_{g,1})^\sigma)$ is dense in any $\cH^{0,\gamma_p})p(\M)$ for all $1\leq p<\infty$. 
Given any $w\in C_c^\infty(\M)\subset L_m(\M)$, take a sequence $(u_k)_k \subset C_c^\infty(\M) \subset D((- \Delta_1)^\sigma)$ converging to $\Phi(w)$ in  $L_1(\M)$. 
Without loss of generality, we may assume that all $u_k$ and $w$ are   supported in some $\Omega\subset\!\subset \M$.
%, for otherwise we can always multiply $u_k$ by a smooth cut-off function $\psi\equiv 1$ in ${\rm supp}(w)$ and vanishes outside $\Omega$. 
One has
$$
\|\beta(u_n)- w\|_1 \leq C \| (u_n-\Phi(w))^{1/m}\|_1 \leq C (\Omega)\| u_n-\Phi(w)\|_1^{1/m}.
$$
%Here the constant $C$ depends only on $\Omega$.
This proves that the closure  of $D(\cA)$ contains $C_c^\infty(\M)$. Since $C_c^\infty(\M) $ is dense in $L_1(\M)$, this shows the density of $D(\cA)$  in $L_1(\M)$.
 
Case 2: $m<1$. In this case, $\beta\in C^{1-}(\R)$.  Given arbitrary $w\in C_c^\infty(\M)$, there exists a  sequence $(u_k)_k \subset C_c^\infty(\M) \subset D((- \Delta_1)^\sigma)$ converging to $\Phi(w) $ in  $L_1(\M)$ satisfying   $\|\beta(u_k)\|_\infty\leq 2\|w\|_\infty$. So the Lipschitz continuity of $\beta$ implies 
\begin{align*}
\|\beta(u_k) - w\|_1 \leq C(w) \| u_k  - \Phi(w)\|_1 .
\end{align*}
As in Case 1, this implies the density of $D(\cA)$ in $L_1(\M)$. 
\end{proof}

%%%%%%%%%%%%%%%%%%%%%%%%%%%%%%%%%%%%%%%%%%%

\section*{Acknowledge}

Parts of the paper were completed while the second author was an assistant professor at Georgia Southern University. 
He would like to thank the staff and faculty at Georgia Southern University for all their support. 

We would like to express our gratitude to Prof. G. Grillo for pointing out to us the recent paper \cite{BerBonGanGri} and other related work on the porous medium equation in negatively curved spaces. 
We also would like to express our
appreciation to the anonymous reviewer for carefully reading the manuscript.

%%%%%%%%%%%%%%%%%%%%%%%%%%%%%%%%%%%%%%%%%%%%%%%%%%%%%%%%%%%%%%%

\end{document}